\documentclass[12pt,letterpaper,leqno]{article}
\usepackage{amsmath}
\usepackage{mathtools}
\usepackage{amsthm}  
\usepackage{amssymb} 
\usepackage{bbm}
\usepackage[nomath,notext]{kpfonts} 
\usepackage{mathrsfs}
\usepackage{enumitem}
\usepackage{bigints}
\usepackage{color}
\usepackage{tikz}  
\usepackage{float}
\usepackage{ifthen}
\usepackage{authblk}

\theoremstyle{definition}

\newtheorem{theorem}{Theorem}[section]
\newtheorem{corollary}[theorem]{Corollary}
\newtheorem{lemma}[theorem]{Lemma}

\newtheorem*{remark}{Remark}

\newcommand{\N}{\ensuremath{\mathbb{N}}}  
\newcommand{\Z}{\ensuremath{\mathbb{Z}}}  
\newcommand{\cN}{\overline{\ensuremath{\mathbb{N}}}}
\newcommand{\R}{\ensuremath{\mathbb{R}}}  
\newcommand{\ma}[1]{\mbox{\ensuremath{#1}}}
\newcommand{\rel}[1]{\overset{#1}{\sim}}  
\newcommand\restr[2]{{
  \left.\kern-\nulldelimiterspace 
  #1 
  \vphantom{\big|} 
  \right|_{#2} 
  }}
\newcommand{\di}[2]{\delta(#1,#2)}
\renewcommand{\P}{\ensuremath{\mathbbm{P}}} 
\newcommand{\E}{\ensuremath{\mathbbm{E}}}  
\newcommand{\V}{\ensuremath{\mathbbm{V}}} 
\newcommand{\Cov}{\ensuremath{\mathbbm{Cov}}}
\newcommand{\X}{\ensuremath{\mathcal{X}}} 
\newcommand{\stp}[1]{\ensuremath{\mathcal{#1}}} 
\newcommand{\given}[1]{\ensuremath{\middle\vert #1}} 
\newcommand{\F}{\mathscr{F}}  
\newcommand{\QVar}[1]{\IP{#1}{#1}} 
\newcommand{\indep}{\perp}
\newcommand{\ds}{\hspace{3pt}ds}
\newcommand{\dif}[1]{\hspace{3pt}d#1}
\newcommand{\der}[2]{\frac{d}{d#2}#1} 
\newcommand{\pder}[2]{\frac{\partial #1}{\partial #2}} 
\newcommand{\bO}[1]{\mathcal{O}\(#1\)} 
\newcommand{\lo}[1]{o\(#1\)} 
\newcommand{\Ind}[1]{\ensuremath{\mathbbm{1}_{#1}}} 
\renewcommand{\S}[1]{ 
\ifx&#1&
\ensuremath{\mathcal{S}}
\else
\ensuremath{\mathcal{S}}\(#1\)
\fi
}
\newcommand{\Bo}[1]{  
\ifx&#1&
\ensuremath{\mathscr{B}}
\else
\ensuremath{\mathscr{B}}(#1)
\fi
}
\renewcommand{\L}[1]{\mathcal{L}^{#1}} 
\newcommand{\cl}[1]{\overline{#1}} 
\newcommand{\Tau}{\ensuremath{\mathcal{T}}} 
\newcommand{\HS}{\mathsf{H}} 
\DeclarePairedDelimiterX{\IP}[2]{\langle}{\rangle}{#1, #2} 
\newcommand{\norm}[1]{\left\lVert #1 \right\rVert} 
\newcommand{\CC}[1]{\overline{#1}} 
\newcommand{\Var}[1]{\ensuremath{Var\(#1\)}}
\newcommand{\Rho}{\ensuremath{\mathcal{P}}} 
\newcommand{\PS}[1]{\mathscr{P}_{#1}} 
\newcommand{\PSC}[2]{\PS{#1}(#2)}
\newcommand{\AF}[1]{\mbox{\ensuremath{\lvert #1 \rvert}}}  
\newcommand{\DAF}[1]{\mbox{\ensuremath{{\lvert #1\rvert}^{\downarrow}}}} 
\newcommand{\DAFi}[2]{\mbox{\ensuremath{{\lvert #1\rvert}^{\downarrow}_{#2}}}} 
\newcommand{\LAF}[1]{\mbox{\ensuremath{\lvert #1 \rvert}}} 
\newcommand{\LAFi}[2]{\mbox{\ensuremath{{\lvert #1 \rvert}_{#2}}}} 
\newcommand{\s}[1]{\sigma(#1)} 
\newcommand{\si}[1]{\sigma^{-1}(#1)} 
\newcommand{\dis}[2]{\delta_m(#1,#2)} 
\newcommand{\Pio}{\Pi^{\downarrow}} 
\newcommand{\pio}{\pi^{\downarrow}}
\newcommand{\PSo}[1]{\mathscr{P}^{\downarrow}_{#1}}
\DeclareMathOperator{\co}{Coag} 
\DeclareMathOperator*{\Co}{CO}  
\newcommand{\tPi}{\widetilde{\Pi}}
\newcommand{\tpi}{\tilde{\pi}}
\newcommand{\mPi}{\mathbb{\Pi}}
\newcommand{\ze}{\mathbf{0}}
\renewcommand{\(}{\left(}
\renewcommand{\)}{\right)}
\renewcommand{\[}{\left[}
\renewcommand{\]}{\right]}
\newcommand{\lbr}{\left\{} 
\newcommand{\rbr}{\right\}} 
\newcommand{\abs}[1]{\left\lvert #1 \right\rvert}

\title{Site Frequency Spectrum of the Bolthausen-Sznitman Coalescent}
\author[1]{G{\"o}tz Kersting}
\author[2]{Arno Siri-J\'egousse} 
\author[2]{Alejandro H. Wences}
\date{}%
\affil[1]{Goethe Universit{\"a}t, Institut f{\"u}r Mathematik, Frankfurt am Main, Germany.} 
\affil[2]{UNAM, IIMAS, Departamento de Probabilidad y Estad\'istica, Mexico.}
\begin{document}
\maketitle
\abstract{We derive explicit formulas for the two first moments of he site frequency spectrum $(SFS_{n,b})_{1\leq b\leq n-1}$ of the Bolthausen-Sznitman coalescent along with some precise and efficient approximations, even for small sample sizes $n$.
These results provide new $L_2$-asymptotics for some values of $b=o(n)$.
We also study the length of internal branches carrying $b>n/2$ individuals. In this case we obtain the distribution function and a convergence in law.
Our results rely on the random recursive tree construction of the Bolthausen-Sznitman coalescent.}

\section{Introduction}
The Bolthausen-Sznitman coalescent is an exchangeable coalescent with multiple collisions that has recently gained attention
in the theoretical population genetics literature. It has been described as the limit process of the genealogies of different 
population evolution models, including models that contemplate the effect of natural selection $\cite{Sch03,Sch16}$. 
It has also been proposed as a new null model for the genealogies of rapidly adapting populations, such as pathogen microbial populations,
and other populations that show departures from Kingman's null model \cite{Des,NH13}.\par

A measure of the genetic diversity in a present day sample of a population is often used in population genetics in order to infer
its evolutionary past and the forces at play in its dynamics. The \textit{Site Frequency Spectrum} (SFS) is a well known 
theoretical model of the genetic diversity present in a population, it assumes that neutral mutations arrive to the population as a 
Poisson Process and that each arriving mutation falls in a different site of the genome (infinite sites model), in contrast to the 
\textit{Allele Frequency Spectrum} in which mutations are assumed to fall on the same site but create a new allele every time 
(infinite alleles model). Given the close relation between the Site Frequency Spectrum and the whole structure of the 
underlying genealogical tree, it can be used as a model selection tool for the evolutionary dynamics of a population \cite{EBBF, Kor, FS}.\par

In this work we give explicit expressions of the first and second moments for the whole
Site Frequency Spectrum $(SFS_{n,b})_{1\leq b<n}$ of the Bolthausen-Sznitman coalescent, which to our knowledge were only known for 
Kingman's coalescent until
now $\cite{Fu95}$. Here $SFS_{n,b}$ denotes the number of mutations shared by $b$ individuals in the
sample of size $n$. For the expectation we obtain the formula
\begin{equation*}
\E\[SFS_{n,b}\] = \theta n \int_0^1 \frac{\Gamma(b-p)}{\Gamma(b+1)} \frac{\Gamma(n-b+p)}{\Gamma(n-b+1)}
\frac{\dif{p}}{\Gamma(1-p)\Gamma(1+p)},
\end{equation*}
where $\theta$ denotes the mutation rate. For larger values of $n$ there might occur problems in the
calculation of this integral due to the exorbitant growth of the Gamma function. Also this formula
allows no insight into the shape of the expected site frequency spectrum. For this purpose
approximations are helpful. A first approximation, resting on Stirling's formula, reads for $2\leq b \leq n-1$

\begin{equation}
\E[SFS_{n,b}] \approx \frac{\theta}{n-1} \frac{b-1}{b} f_1\(\frac{b-1}{n-1}\)
\label{approx}
\end{equation}
where $f_1$ is a convex, non-monotone function on $(0,1)$
defined by
\begin{equation} 
\label{eq:f1}
f_1(u)\coloneqq \int_0^1 u^{-p-1}(1-u)^{p-1} \frac{\sin (\pi p)}{\pi p} \dif{p} \hspace{5pt}.
\end{equation}
We remark that this integral may be reduced to the (complex) exponential integral $Ei(\cdot)$. These formulas show that the shape of the Site Frequency Spectrum, restricted to the range
$2 \leq  b < n$, is explained essentially by one function not depending on the population size $n$. Also our approximations update those
given in $\cite{NH13}$ for the case of families with frequencies close to 0 and 1, since we have
$f_1(u){\sim} ({u\log u})^{-2}$ close to 0 and $f_1(u){\sim} ({(u-1)\log(1-u)})^{-1}$ close to 1, see equations \eqref{f1at0} and \eqref{f1at1} below. The case $b = 1$ is not covered by \eqref{approx}, 
it has to be treated separately, which reflects the dominance of external branches in the Bolthausen-Sznitman coalescent. See Theorem $\ref{cor:ASMEV_SFSI}$ for a complete summary. \par

\begin{figure}[h]
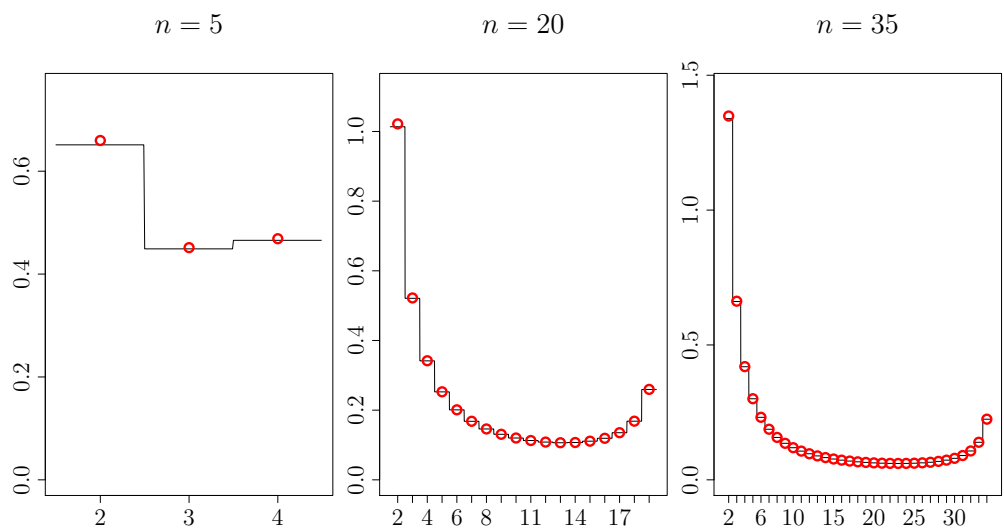

\centering
\scalebox{0.75}{

}
\caption{Comparison of exact and approximated values of $\E[SFS_{n,b}]$, red circles present the exact values for $b=2$ to $n-1$,
and the black lines their refined approximations \eqref{approx2}.} 
\label{fig:ESFScompIntro}
\end{figure}\par

The above approximation is accurate also from a numerical point of view. Only for $b = 2$
we encounter an enlarged relative error which anyhow remains less than 10 percent for $n \geq 8$. If a more precise result is desired then 
the following refined approximation may be applied for $2\leq b\leq n$:
\begin{equation}\label{approx2}
\E\[SFS_{n,b}\] \approx \theta n\frac{b-1}{b} \( \frac{1}{(n-1)^2}f_1\(\frac{b-1}{n-1}\) - \frac{1}{(n-1)^3} g_1\(\frac{b-1}{n-1}\) \),
\end{equation}
with a positive function $g_1$ on $(0,1)$ given by
\begin{equation}\label{eq:g1}
g_1(u)\coloneqq \frac{1}{2u^2(1-u)^2} \frac{\pi^2 + \log^2 \frac{1-u}{u} + \frac{2}{u} \log \frac{1-u}{u}}{\(\pi^2 + \log^2 \frac{1-u}{u}\)^2}.
\end{equation}
With this formula we have a relative error remaining below 1 percent for $b = 2$ and $n \geq 10$,
below 0.5 percent for $b = 2$ and $n\geq 150$, and below 0.3 percent for $b \geq 3$ and $n \geq 10$. Thus this
approximation appears well-suited for practical purpose. Figure $\ref{fig:ESFScompIntro}$ illustrates its precision in
the cases $n = 5$, $20$, $35$ and $\theta = 1$.\par

For $b=1$ the approximation formula corresponding to \eqref{approx} reads
\begin{align*}
\E\[SFS_{n,1}\] &= \theta n \int_0^1  \frac{\Gamma(n-1+p)}{\Gamma(n)}
\frac{\dif{p}}{\Gamma(1+p)} \\&\approx \theta n \int_0^1  (n-1)^{p-1}
\frac{\dif{p}}{\Gamma(1+p)},
\end{align*}
which is an immediate consequence of Stirling's approximation. It is precise for small $n$ and requires no further correction as in the case $b\ge 2$.\par

We also study the asymptotic behavior of the second moments which, together with the above asymptotics for the first moment, leads to the 
following $L^2$ convergences:
\begin{equation*}
\frac{\log n}{n} SFS_{n,1} \to \theta,
\end{equation*}
and, whenever $b\geq 2$ and $b=\lo{\sqrt{n}/\log n}$,
\begin{equation*}
\frac{b(b-1)\log^2\(n/b\)}{n} SFS_{n,b} \to \theta.
\end{equation*}
These generalize and strengthen the results in $\cite{DK}$ for the Bolthausen-Sznitman coalescent.\par

We also provide the joint distribution function of the branch lengths of large families, 
i.e families of size at least half the total population size,
and their marginal distribution function. These results are useful to obtain the marginal distribution function of the Site Frequency Spectrum and a sampling formula for the half of the vector 
corresponding to large family sizes, although we do not present such tedious
computations here.\par 

Asymptotic results for related functionals on the Bolthausen-Sznitman coalescent have been derived by studying the block count chain of the coalescent
through a coupling with a random walk as in $\cite{DIMR07}$ and $\cite{KPSJ14}$, where asymptotics for the total number of jumps, 
and the total, internal, and external branch lengths of the Bolthausen-Sznitman coalescent are described; these results give the
asymptotic behaviour of the total number of mutations present in the population, the number of mutations present in a single individual,
and the number of mutations present in at least 2 individuals. Also, a
Markov chain approximation of the initial steps of the process was developed in $\cite{DK}$ 
where asymptotics for the total tree length and the Site Frequency Spectrum of small families were derived 
for a class of $\Lambda$-coalescents containing the Bolthausen-Sznitman coalescnet. \par

Progress has also been made for the finite coalescent even for the general coalescent process. The finite Bolthausen-Sznitman coalescent 
has been studied through 
the spectral decomposition of its jump rate matrix described
in $\cite{KuklaPitters15}$ where the authors used it to derive explicit expressions for the transition probabilities and the Green's 
matrix of this coalescent, and also the Kingman 
coalescent. The spectral decomposition of the jump rate matrix of a general coalescent, including coalescents with multiple mergers, is also used in 
$\cite{SKS16}$ where an expression for the expected Site Frequency Spectrum is given in terms of matrix operations which in the case of 
the Bolthausen-Sznitman coalescent result in an algorithm requiring on the order of $n^2$ computations. In $\cite{HJB19}$ another expression 
in terms of matrix operations is given for this and other functionals on general coalescent processes, 
both in expected value (and higher moments) and in distribution; these expressions however are deduced from the theory of phase-type distributions, 
in particular distributions of rewards constructed on top of coalescent processes, and also require vast computations for large population sizes. \par

Our method, mainly based on the Random Recursive Tree
construction of the Bolthausen-Sznitman coalescent given in $\cite{GM05}$, gives easy-to-compute expressions for the first and second
moments of the Site Frequency Spectrum of this particular coalescent. This combinatorial construction not only allows us to 
study the bottom but also the top of the tree thus providing an additional insight into the past of the population 
and large families, both asymptotically and for any fixed population size. \par

In Section $\ref{sec:Preliminaries}$ we layout the basic intuitions that compose the bulk of our method, including the Random Recursive Tree
construction of the Bolthausen-Sznitman coalescent and the derivation of the first moment of the Site Frequency Spectrum for the infinite 
coalescent as a first application (Corollary $\ref{cor:EVSFSinfty}$). In Section $\ref{sec:BLmoments}$ we present our results on the first and
second moments of the branch lengths (Theorem $\ref{th:Elnb}$) and of the Site Frequency Spectrum (Corollary $\ref{cor:SFS_EvCov}$) for any fixed family size and initial population. We then
use these expressions to obtain asymptotic approximations of these moments as the initial population goes to infinity (Theorems $\ref{cor:ASMEV_SFSI}$ and $\ref{th:AsmpCov}$) which lead to $L^2$ convergence results on the SFS (Corollary $\ref{lawlargenumbers}$). 
In Section $\ref{sec:BLdistribution}$ we restrict ourselves to the case of large family sizes and present the joint and marginal distribution functions of their branch lengths (Theorems $\ref{Prop:Llnb}$ and $\ref{le:JointLlnb}$), 
along with a limit in law result (Corollary $\ref{cor:ConvDistlnb}$). Section \ref{sec:approximations} provides explanations for approximations \eqref{approx} and \eqref{approx2}.
Finally, in Sections
$\ref{sec:proofsBLmoments}$ and \ref{sec:proofsBLdistribution} we provide detailed proofs of our results.


\section{Preliminaries} \label{sec:Preliminaries}

Consider the Bolthausen-Sznitman coalescent $(\Pi^\infty(t))_{t\geq0}$ with values in $\PS{\infty}$, the space of partitions of $\N$,
and the ranked coalescent $(\DAF{\Pi^\infty(t)})_{t\geq 0}$, with values in the space of mass partitions $\PS{[0,1]}$, made of the asymptotic frequencies of $\Pi^\infty(t)$ reordered in a non-increasing way.
In what follows we present the Random Recursive Tree (RRT) construction of the Bolthausen-Sznitman coalescent given
by Goldschmidt and Martin in \cite{GM05}; then we follow the argument given in the same paper to establish that
\begin{equation}\label{PDlaw}\DAF{\Pi^\infty(t)}\overset{d}{=}PD(e^{-t},0),\end{equation}
where $PD(\alpha,\theta)$ is the $(\alpha,\theta)-$Poisson-Dirichlet distribution.
\par

Briefly, the construction of the Bolthausen-Sznitman coalescent in terms of Random Recursive Trees proceeds as follows. 
We work on the set of recursive trees whose 
labeled nodes form a partition $\pi$ of $[n]\coloneqq \{1,\dots, n\}$, where the ordering of the nodes that confers the term ``recursive''
is given by ordering the blocks of $\pi$ according to their least elements. A cutting-merge procedure is defined
on the set of recursive trees of this form with a marked edge, this procedure consists of cutting the marked 
edge and merging all the labels in the subtree below with the node above, thus creating a new recursive tree whose labels
form a new (coarser) partition of $[n]$ (see Figure $\ref{fig:CuttingMerge}$). With this operation in mind 
we consider a RRT with labels $\{1\},\cdots,\{n\}$, 
say $T$, to which we also attach
independent standard exponential variables to each edge. Then, for each time $t>0$ 
we retrieve the partition of $[n]$ obtained 
by performing a cutting-merge procedure on all the edges of $T$ whose exponential variable is less than $t$. This
gives a stochastic process $(\Pi^n(t))_{t\geq 0}$ with values on the set of partitions of $[n]$ that can be proven to
be the $n$-Bolthausen-Sznitman coalescent. 

\begin{figure}[h]
\centering
\scalebox{0.8}{
\begin{tikzpicture}[x=1pt,y=1pt]
\definecolor{fillColor}{RGB}{255,255,255}
\path[use as bounding box,fill=fillColor,fill opacity=0.00] (0,0) rectangle (469.75,238.49);
\begin{scope}
\path[clip] (  0.00,  0.00) rectangle (469.75,238.49);
\definecolor{drawColor}{RGB}{0,0,0}

\node[text=drawColor,anchor=base,inner sep=0pt, outer sep=0pt, scale=  1.00] at (117.44,223.12) {\{1,3\}};

\node[text=drawColor,anchor=base,inner sep=0pt, outer sep=0pt, scale=  1.00] at ( 82.21,151.58) {\{4\}};

\node[text=drawColor,anchor=base,inner sep=0pt, outer sep=0pt, scale=  1.00] at (223.13,151.58) {\{2\}};

\node[text=drawColor,anchor=base,inner sep=0pt, outer sep=0pt, scale=  1.00] at ( 11.74, 80.03) {\{5,7\}};

\node[text=drawColor,anchor=base,inner sep=0pt, outer sep=0pt, scale=  1.00] at ( 82.21, 80.03) {\{6\}};

\node[text=drawColor,anchor=base,inner sep=0pt, outer sep=0pt, scale=  1.00] at (152.67, 80.03) {\{9\}};

\node[text=drawColor,anchor=base,inner sep=0pt, outer sep=0pt, scale=  1.00] at ( 82.21,  8.48) {\{8,10\}};

\path[draw=drawColor,line width= 0.8pt,line join=round,line cap=round] (117.44,215.83) --
	( 82.21,165.75);

\path[draw=drawColor,line width= 0.8pt,line join=round,line cap=round] (117.44,215.83) --
	(223.13,165.75);

\path[draw=drawColor,line width= 0.8pt,line join=round,line cap=round] ( 82.21,144.29) --
	( 11.74, 94.20);

\path[draw=drawColor,line width= 0.8pt,dash pattern=on 4pt off 4pt ,line join=round,line cap=round] ( 82.21,144.29) --
	( 82.21, 94.20);

\path[draw=drawColor,line width= 0.8pt,line join=round,line cap=round] ( 82.21,144.29) --
	(152.67, 94.20);

\path[draw=drawColor,line width= 0.8pt,line join=round,line cap=round] ( 82.21, 72.74) --
	( 82.21, 22.66);
\end{scope}
\begin{scope}
\path[clip] (  0.00,  0.00) rectangle (469.75,238.49);
\definecolor{drawColor}{RGB}{0,0,0}

\node[text=drawColor,anchor=base,inner sep=0pt, outer sep=0pt, scale=  1.00] at (352.32,223.12) {\{1,3\}};

\node[text=drawColor,anchor=base,inner sep=0pt, outer sep=0pt, scale=  1.00] at (317.08,151.58) {\{4,6,8,10\}};

\node[text=drawColor,anchor=base,inner sep=0pt, outer sep=0pt, scale=  1.00] at (458.01,151.58) {\{2\}};

\node[text=drawColor,anchor=base,inner sep=0pt, outer sep=0pt, scale=  1.00] at (246.62, 80.03) {\{5,7\}};

\node[text=drawColor,anchor=base,inner sep=0pt, outer sep=0pt, scale=  1.00] at (387.55, 80.03) {\{9\}};

\path[draw=drawColor,line width= 0.8pt,line join=round,line cap=round] (352.32,215.83) --
	(317.08,165.75);

\path[draw=drawColor,line width= 0.8pt,line join=round,line cap=round] (352.32,215.83) --
	(458.01,165.75);

\path[draw=drawColor,line width= 0.8pt,line join=round,line cap=round] (317.08,144.29) --
	(246.62, 94.20);

\path[draw=drawColor,line width= 0.8pt,line join=round,line cap=round] (317.08,144.29) --
	(387.55, 94.20);
\end{scope}
\end{tikzpicture}
}
\caption{On the left, an example of a recursive tree whose labels constitute a partition of $\{1,\cdots,10\}$.
On the right, the resulting recursive tree after a cutting-merge procedure performed on the marked edge (dashed line) of the
first tree.}
\label{fig:CuttingMerge}
\end{figure}
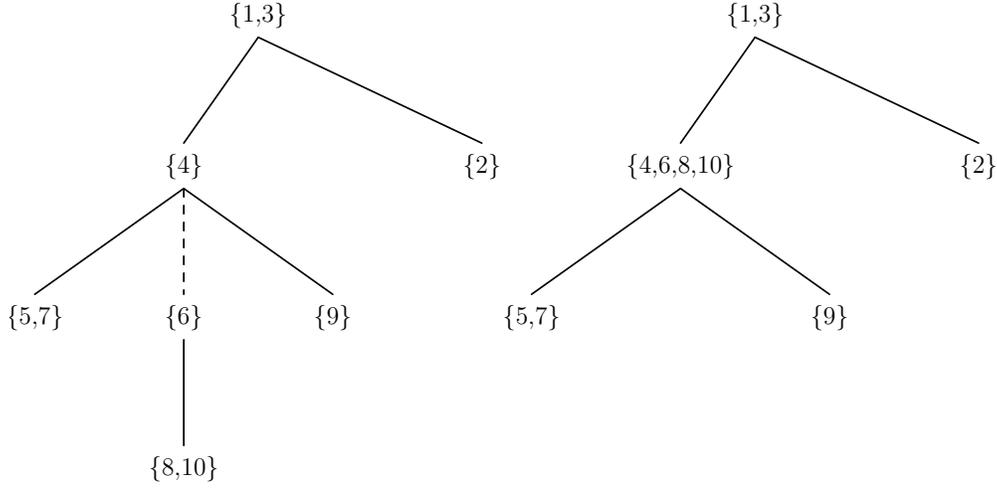\par

The fact that $\DAF{\Pi^\infty(t)}\overset{d}{=}PD(e^{-t},0)$ now follows readily. To see this, consider
the construction of $T$ where nodes arrive sequentially and each arriving node attaches to any of the previous nodes 
with equal probability. Considering also their exponential edges and having in mind
the cutting-merge procedure we see that for any fixed time $t$, and assuming that $b-1$ nodes have arrived and
formed $k$ blocks of sizes $s_1,\dots,s_k$ in $\Pi^{b-1}(t)$, 
the next arriving node, node $\{b\}$, will form a new block in $\Pi^b(t)$
if and only if it attaches to any of the roots of the sub-trees of $T$ that form the said $k$ blocks  and if, 
furthermore,
its exponential edge is greater than $t$; this occurs with probability $\frac{ke^{-t}}{b-1}$. On the other hand,
in order for $\{b\}$ to join the $jth$ block of size $s_j$ it must either attach to the root of the sub-tree 
of $T$ that builds 
this block and its exponential edge must be less than $t$, which happens with probability $\frac{1-e^{-t}}{b-1}$, or
it must attach to any other node of the said sub-tree, which happens with probability $\frac{s_j-1}{b-1}$; thus, the
probability of attaching to the $jth$ block is $\frac{s_j-e^{-t}}{b-1}$. 
We recognize in these expressions the probabilities
that define the Chinese Restaurant Process with parameters $\alpha=e^{-t}$ and $\theta=0$.

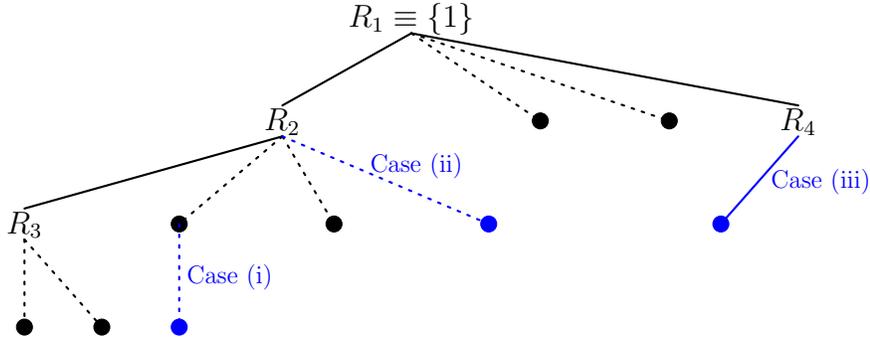
\begin{figure}[h]
\centering
\hspace*{-1.8cm}
\begin{tikzpicture}[x=1pt,y=1pt]
\definecolor{fillColor}{RGB}{255,255,255}
\path[use as bounding box,fill=fillColor,fill opacity=0.00] (0,0) rectangle (361.35,144.54);
\begin{scope}
\path[clip] (  0.00,  0.00) rectangle (361.35,144.54);
\definecolor{drawColor}{RGB}{0,0,0}

\node[text=drawColor,anchor=base,inner sep=0pt, outer sep=0pt, scale=  1.00] at (180.68,127.37) {$R_1\equiv\{1\}$};

\node[text=drawColor,anchor=base,inner sep=0pt, outer sep=0pt, scale=  1.00] at (131.89, 88.34) {$R_2$};

\node[text=drawColor,anchor=base,inner sep=0pt, outer sep=0pt, scale=  1.00] at (327.02, 88.34) {$R_4$};
\definecolor{fillColor}{RGB}{0,0,0}

\path[draw=drawColor,line width= 0.4pt,line join=round,line cap=round,fill=fillColor] (229.46, 91.78) circle (  3.01);

\path[draw=drawColor,line width= 0.4pt,line join=round,line cap=round,fill=fillColor] (278.24, 91.78) circle (  3.01);

\node[text=drawColor,anchor=base,inner sep=0pt, outer sep=0pt, scale=  1.00] at ( 34.33, 49.31) {$R_3$};

\path[draw=drawColor,line width= 0.4pt,line join=round,line cap=round,fill=fillColor] ( 92.87, 52.76) circle (  3.01);

\path[draw=drawColor,line width= 0.4pt,line join=round,line cap=round,fill=fillColor] (151.41, 52.76) circle (  3.01);
\definecolor{drawColor}{RGB}{0,0,255}
\definecolor{fillColor}{RGB}{0,0,255}

\path[draw=drawColor,line width= 0.4pt,line join=round,line cap=round,fill=fillColor] (209.94, 52.76) circle (  3.01);

\path[draw=drawColor,line width= 0.4pt,line join=round,line cap=round,fill=fillColor] (297.75, 52.76) circle (  3.01);

\path[draw=drawColor,line width= 0.4pt,line join=round,line cap=round,fill=fillColor] ( 92.87, 13.73) circle (  3.01);
\definecolor{drawColor}{RGB}{0,0,0}
\definecolor{fillColor}{RGB}{0,0,0}

\path[draw=drawColor,line width= 0.4pt,line join=round,line cap=round,fill=fillColor] ( 34.33, 13.73) circle (  3.01);

\path[draw=drawColor,line width= 0.4pt,line join=round,line cap=round,fill=fillColor] ( 63.60, 13.73) circle (  3.01);

\path[draw=drawColor,line width= 0.8pt,line join=round,line cap=round] (180.68,124.95) --
	(131.89, 97.64);

\path[draw=drawColor,line width= 0.8pt,line join=round,line cap=round] (180.68,124.95) --
	(327.02, 97.64);

\path[draw=drawColor,line width= 0.8pt,dash pattern=on 1pt off 3pt ,line join=round,line cap=round] (180.68,124.95) --
	(229.46, 91.78);

\path[draw=drawColor,line width= 0.8pt,dash pattern=on 1pt off 3pt ,line join=round,line cap=round] (180.68,124.95) --
	(278.24, 91.78);

\path[draw=drawColor,line width= 0.8pt,line join=round,line cap=round] (131.89, 85.93) --
	( 34.33, 58.61);

\path[draw=drawColor,line width= 0.8pt,dash pattern=on 1pt off 3pt ,line join=round,line cap=round] (131.89, 85.93) --
	( 92.87, 52.76);

\path[draw=drawColor,line width= 0.8pt,dash pattern=on 1pt off 3pt ,line join=round,line cap=round] (131.89, 85.93) --
	(151.41, 52.76);
\definecolor{drawColor}{RGB}{0,0,255}

\path[draw=drawColor,line width= 0.8pt,dash pattern=on 1pt off 3pt ,line join=round,line cap=round] (131.89, 85.93) --
	(209.94, 52.76);

\node[text=drawColor,anchor=base,inner sep=0pt, outer sep=0pt, scale=  0.80] at (182.63, 72.44) {Case (ii)};

\path[draw=drawColor,line width= 0.8pt,line join=round,line cap=round] (327.02, 85.93) --
	(297.75, 52.76);

\node[text=drawColor,anchor=base west,inner sep=0pt, outer sep=0pt, scale=  0.80] at (316.78, 66.59) {Case (iii)};

\path[draw=drawColor,line width= 0.8pt,dash pattern=on 1pt off 3pt ,line join=round,line cap=round] ( 92.87, 52.76) --
	( 92.87, 13.73);

\node[text=drawColor,anchor=base west,inner sep=0pt, outer sep=0pt, scale=  0.80] at ( 95.79, 30.49) {Case (i)};
\definecolor{drawColor}{RGB}{0,0,0}

\path[draw=drawColor,line width= 0.8pt,dash pattern=on 1pt off 3pt ,line join=round,line cap=round] ( 34.33, 46.90) --
	( 34.33, 13.73);

\path[draw=drawColor,line width= 0.8pt,dash pattern=on 1pt off 3pt ,line join=round,line cap=round] ( 34.33, 46.90) --
	( 63.60, 13.73);
\end{scope}
\end{tikzpicture}
\caption{Schematic representation of passing from $\Pi^n(t)$ to $\Pi^{n+1}(t)$ for fixed $t$, by adding a new node (blue) to a RRT. 
Solid lines and dotted lines represent edges whose exponential variables are greater than $t$ and less than or equal to $t$, respectively. 
In this case at time $t$ there are four subtrees rooted at $R_1,R_2,R_3,$ and $R_4$, which form the blocks that 
constitute $\Pi^n(t)$; these blocks are also the tables of a Chinese Restaurant Process. In case (i) the new node will be included
in the block formed by $R_2$ at time $t$, irrespective of whether its exponential edge is greater than $t$ or not. In case (ii) the new
node forms part of the block rooted at $R_4$ because its exponential edge is less than $t$. Finally, in case (iii) the new
node is a new root of a subtree that will form an additional block of $\Pi^{n+1}(t)$ (i.e. the new node opens a new table in the
Chinese Restaurant Process).}
\label{fig:RRT-BSC}
\end{figure}\par

We now provide two straightforward applications of the RRT construction described above which nonetheless contain
the essential intuitions underlying the forthcoming proofs.

\subsection{Site Frequency Spectrum in the infinite coalescent}
For the first application consider a subset $I\subset(0,1)$ and define $(C_I(t))_{t\geq0}$ to be the process of the number
of blocks in $\Pi^\infty(t)$ with asymptotic frequencies in $I$. Then
\begin{equation}\label{def:ell_I}
	\ell_I\coloneqq\int_0^\infty C_I(t) \dif{t}
\end{equation}
gives the total branch length of families with size frequencies in $I$ in the infinite coalescent.
\par
Our first theorem is a simple corollary of the equality in law \eqref{PDlaw}.
\begin{theorem}\label{th:EVSFSinfty}
For $I\subset (0,1)$, we have 
\begin{equation*}
\E[\ell_{I}]=\int_I \int_0^1u^{-p-1}(1-u)^{p-1}\frac{\sin(\pi p)}{\pi p} \dif{p} \dif{u}.
\end{equation*}
\end{theorem}
In particular, note that if in the infinite sites model with mutation rate $\theta$ we define $SFS_I$ to be the number of mutations 
shared by a proportion $u$ of individuals with $u$ ranging in $I$,
then by conditioning on $\ell_I$ we get 
\begin{corollary}\label{cor:EVSFSinfty}
For $I\subset (0,1)$, we have 
\begin{equation}\label{ESFSI}
\E\[SFS_I\]=\theta \int_I \int_0^1u^{-p-1}(1-u)^{p-1}\frac{\sin(\pi p)}{\pi p} \dif{p} \dif{u}.
\end{equation}
\end{corollary}
\begin{proof}[Proof of Theorem $\ref{th:EVSFSinfty}$]
Since 
\begin{equation*}
\E\[\ell_I\]=\int_0^\infty\E\[C_I(t)\]\dif{t}
\end{equation*}
it only remains to compute $\E\[C_I(t)\]$ and simplify the expressions, but this is a straightforward consequence of Equation (6) in \cite{PY77} which states that if $\varrho=(a_1,\cdots)$ is $PD(\alpha,\theta)$ distributed, and $f:\R\to\R$ is a function, then
	\begin{equation}\label{Le:EVPD}
		\E\bigg[\sum_{i=1}^\infty f(a_i)\bigg]=
		\frac{\Gamma(\theta+1)}{\Gamma(\theta+\alpha)\Gamma(1-\alpha)}
		\int_0^1 f(u) \frac{(1-u)^{\alpha+\theta-1}}{u^{\alpha+1}} \dif{u}.
	\end{equation}
Taking $f(u)= \Ind{I}(u)$
we get
\begin{equation*}
	\E[C_{I}(t)]=\frac{1}{\Gamma(e^{-t})\Gamma(1-e^{-t})}
	                                   \int_0^1 \Ind{I}(u) \frac{(1-u)^{e^{-t}-1}}{u^{e^{-t}+1}} \dif{u}.
\end{equation*}
Using Euler's reflection formula, making $p=e^{-t}$ on the above expression and integrating on $[0,\infty)$ we finish the proof.
\end{proof}
\subsection{Time to the absorption}
In this section we prove a useful lemma for the upcoming proofs, but a first consequence of this lemma gives
the distribution function of the time to absorption, $A_n$, in the $n$-coalescent, a result already proved in 
$\cite{MohlePitters14}$.

Here $Be$ stands for the Beta
function
\begin{equation*}
Be(x,y)=\frac{\Gamma(x)\Gamma(y)}{\Gamma(x+y)},
\end{equation*}
and $\Psi$ for the digamma function
\begin{equation*}
\Psi(x)=\frac{\Gamma'(x)}{\Gamma(x)} = -\gamma - \sum_{n=1}^\infty\(\frac{1}{z+n-1}-\frac{1}{n}\)
\end{equation*}
where $\gamma$ stands for the Euler-Mascheroni constant.


\begin{lemma}\label{le:lawsmTMT}
Let $T$ be a RRT on a set of $n$ labels and with exponential edges. Define the two functionals 
$m(T)$ and $M(T)$
 that give the minimum and the maximum of the exponential edges attached to the root of $T$.
Then
\begin{equation}\label{eq:lawmT}
\P(m(T)>s)=\frac{1}{(n-1)Be(n-1,e^{-s})},
\end{equation}
and
\begin{equation}\label{eq:lawMT}
\P(M(T)\leq s)=\frac{1}{(n-1)Be(n-1,1-e^{-s})}.
\end{equation}
Also, for independent trees $T_1$ and $T_2  $ of respective size $n_1$ and $n_2$, we have
\begin{align}\label{eq:lawmT-MT}
&\P(m(T_2)-M(T_1)>s)\nonumber\\=&\frac{1}{(n_1-1)(n_2-1)}\int_0^1 \frac{\Psi(n_1-p)-\Psi(1-p)}{Be(n_2-1,e^{-s}p)Be(n_1-1,1-p)} \dif{p}.
\end{align}
\end{lemma}

The proof of $\eqref{eq:lawMT}$ follows the same lines as in $\cite{MohlePitters14}$ where the law of the time to absorption
of the Bolthausen-Sznitman coalescent is derived, since this time is the maximum of the exponential edges attached to the root of a RRT. That is,  
\begin{equation}\label{eq:DistTreeHeight}
\P(A_n\leq s)=\frac{1}{(n-1)Be(n-1,1-e^{-s})},
\end{equation}
and, as $n\to\infty$,
\begin{equation}\label{cvA}
A_n-\log\log n \overset{d}{\to}-\log E
\end{equation}
where $E$ is a standard exponential random variable. The latter convergence in distribution was elegantly proved in \cite{GM05} using 
a construction of random recursive trees in continuous time, whereas in this case it follows from Stirling's approximation to the 
Gamma functions appearing in $\eqref{eq:DistTreeHeight}$.

On the other hand, the equality $\eqref{eq:lawmT-MT}$ will be used in the computation of the distribution function 
of branch lengths with large family sizes presented in Section \ref{sec:BLdistribution}.

\begin{proof}[Proof of Lemma $\ref{le:lawsmTMT}$]
Let $E_2,\cdots,E_n$ be the exponential edges associated to the nodes of $T$.
For the proof of $\eqref{eq:lawmT}$ we consider the event $\{m(T)>s\}$.
This event occurs when, in the recursive construction of $T$ along with the exponential edges,
the $i$th node ($2\leq i \leq n$)
does not attach to $\{1\}$ whenever $E_i< s$;
this happens with probability $1-\frac{1-e^{-s}}{i-1}$. Thus, considering the $n$ nodes, we obtain
\begin{align*}
\P(m(T)>s)&= e^{-s}\bigg(\frac{1+e^{-s}}{2}\bigg)\dots\bigg(\frac{n-2+e^{-s}}{n-1}\bigg)\\
&=\frac{1}{(n-1)Be(n-1,e^{-s})}.
\end{align*}
For $\eqref{eq:lawMT}$ we instead build the tree such that the $i$th node does not attach
to $\{1\}$ whenever $E_i>s$; this happens with probability $1-\frac{e^{-s}}{i-1}$. Thus we obtain
\begin{align*}
\P(M(T)\leq s)&=(1-e^{-s})\bigg(\frac{2-e^{-s}}{2}\bigg)\dots\bigg(\frac{n-1-e^{-s}}{n-1}\bigg)\\
&=\frac{1}{(n-1)Be(n-1,1-e^{-s})}.
\end{align*}
Finally we compute
\begin{align*}
&\P(m(T_2)-M(T_1)>s)\\
=&\frac{1}{(n_1-1)(n_2-1)}\int_0^\infty \frac{1}{Be(n_2-1,e^{-(s+t)})} \der{\left(\frac{1}{Be(n_1-1,1-e^{-t})}\right)}{t}
   \dif{t}
\end{align*}
and by changing the variable $p=e^{-x}$ we obtain $\eqref{eq:lawmT-MT}$.
\end{proof}


\section{Moments of the Site Frequency Spectrum} \label{sec:BLmoments}
By a simple adaptation of our previous notation for branch lengths in the infinite coalescent ($C_I$ and $\ell_I$), 
in the finite case we also define for $1\leq b \leq n-1$
the process
$(C_{n,b}(t))_{t\geq 0}$ and the random variables $(\ell_{n,b})$, where $C_{n,b}(t)$ is the number of blocks of size $b$ in $\Pi^n(t)$,
and 
\begin{equation}\label{defellCn}
\ell_{n,b}\coloneqq\int_0^\infty C_{n,b}(t) \dif{t}.
\end{equation} We now provide explicit expressions for
$\E\[\ell_{n,b}\]$ and $\E\[\ell_{n,b_1}\ell_{n,b_2}\]$; for this we define the functions
\begin{equation*}
L_1(n,b)= \int_0^1 \frac{\Gamma(b-p)}{\Gamma(b+1)} \frac{\Gamma(n-b+p)}{\Gamma(n-b+1)}
\frac{\dif{p}}{\Gamma(1-p)\Gamma(1+p)},
\end{equation*}
\begin{align*}
 L_2(n,b_1,b_2)&=
\int_0^1\int_0^{p_1}\frac{\Gamma(b_1-p_1)}{\Gamma(b_1+1)}\frac{\Gamma(b_2-b_1+p_1-p_2)}{\Gamma(b_2-b_1+1)}\\&\times
             \frac{\Gamma(n-b_2+p_2)}{\Gamma(n-b_2+1)}\frac{\dif p_2 \dif p_1}{p_1\Gamma(1-p_1)\Gamma(p_1-p_2)\Gamma(p_2+1)}
\end{align*}
and
\begin{align*}
L_3(n,b_1,b_2)&=
\int_0^1\int_0^1 \frac{\Gamma(b_1-p_1)}{\Gamma(b_1+1)}\frac{\Gamma(b_2-p_2)}{\Gamma(b_2+1)}
\\&\times\frac{\Gamma(n-b_1-b_2+p_1+ p_2)}{\Gamma(n-b_1-b_2+1)} 
\frac{\dif p_2\dif p_1}{\Gamma(1-p_1)\Gamma(1-p_2)(p_1\vee p_2)\Gamma(p_1 + p_2)}.
\end{align*}

\begin{theorem}\label{th:Elnb}
For any pair of integers $n,b$ such that $1 \leq b \leq n-1$, we have
\begin{equation}\label{eq:Elnb}
\E[\ell_{n,b}]=n L_1(n,b)
\end{equation}
Also, for any triple of integers $n,b_1,b_2$, with $1\leq b_1 \leq b_2 \leq n-1$, we have
\begin{equation}\label{eq:Elb1lb2}
\E\[\ell_{n,b_1}\ell_{n,b_2}\]=
nL_2(n,b_1,b_2)+nL_3(n,b_1,b_2)\Ind{\{b_1+b_2\leq n\}}
\end{equation}
\end{theorem}
As before, we may define $SFS_{n,b}$ as the number of mutations shared by $b$ individuals in the $n$-coalescent. 
By conditioning on the value of the associated branch lengths we get 
\begin{corollary}\label{cor:SFS_EvCov}
For $1\leq b\leq n-1$,
\begin{equation*}
\E[SFS_{n,b}]=\theta n L_1(n,b)
\end{equation*}
and, for $1\leq b_1 \leq b_2 \leq n-1$, we have,
\begin{align*}
\Cov\(SFS_{n,b_1},SFS_{n,b_2}\)=&\theta^2nL_2(n,b_1,b_2)+\theta^2nL_3(n,b_1,b_2)\Ind{b_1+b_2\leq n}\\&-\theta^2n^2L_1(n,b_1)L_1(n,b_2)+
\theta nL_1(n,b) \Ind{b_1=b=b_2}.
\end{align*}
 \end{corollary}
We also characterize the asymptotic behavior of the functions $L_1,L_2$ and $L_3$ as $n\to\infty$, which in turn give asymptotic 
approximations for the first and second moments of the branch lengths and of $SFS$. For this we recall the function $f_1$ defined in 
$\eqref{eq:f1}$
and also define for $0<u_1<u_2<1$,
\begin{equation}\label{eq:f2}
      f_2(u_1,u_2)\coloneqq \int_0^1\int_0^{p_1} \frac{u_1^{-p_1-1}
                              (u_2-u_1)^{p_1-p_2-1}
                              \(1-u_2\)^{p_2-1}}
                              {p_1\Gamma(1-p_1)\Gamma(p_1-p_2)\Gamma(p_2+1)}\dif{p_2}\dif{p_1}, \hspace{5pt}
\end{equation}
and, for $u_1,u_2>0, u_1+u_2< 1$,
\begin{equation}\label{eq:f3}
    f_3(u_1,u_2)\coloneqq\int_0^1\int_0^1 \frac{u_1^{-p_1-1}u_2^{-p_2-1}
                                \(1-u_1-u_2\)^{p_1+p_2-1}}
                        {\Gamma(1-p_1)\Gamma(1-p_2)(p_1\vee p_2)\Gamma(p_1 + p_2)}\dif{p_2}\dif{p_1} \hspace{5pt}.
\end{equation}

\begin{lemma}\label{le:ASMPL1L2L3}
We have as $n\to\infty$,
\begin{equation}\label{eq:ASMPL1}
\max_{2\leq b \leq n-1} \abs{\frac{n^2L_1(n,b)}{f_1\(\frac{b-1}{n-1}\)} - \frac{b-1}b} \to 0,
\end{equation}
whereas for  $b=1$,
\begin{equation}\label{eq:ASM_EV_SFS1}
\frac{n^2}{(\log n) f_1\(\frac{1}{n-1}\)}L_1(n,1)\to 1.
\end{equation}
Similarly
\begin{equation}\hspace{-30pt}\label{eq:ASMPL2}
\max_{2\leq b_1 < b_2\leq n-1}\abs{\frac{n^3L_2(n,b_1,b_2)}{f_2\(\frac{b_1-1}{n-1},\frac{b_2-1}{n-1}\)}-
                                      \frac{b_1-1}{b_1}} \to 0,
\end{equation}
and if also $b_1 \vee (n-b_2)\to \infty$ then
\begin{equation}\hspace{-30pt}\label{eq:ASMPL3}
	\max_{\substack{2\leq b_1 \leq b_2\leq n-1 \\ b_1+b_2 < n}}
      \abs{\frac{n^3L_3(n,b_1,b_2)}{f_3\(\frac{b_1-1}{n-2},\frac{b_2-1}{n-2}\)}-
                                    \(\frac{b_1-1}{b_1}\)\(\frac{b_2-1}{b_2}\)}\to 0.
\end{equation}
\end{lemma}
\begin{remark}
The above lemma does not cover the cases $b_1=1$ or $b_1=b_2$ for $L_2$, nor the cases $b_1 = 1$, $b_2=1$, $n=b_1+b_2$
or $b_1 \vee (n-b_2) \not\to\infty$ for $L_3$. However, using the same techniques we also obtain asymptotics
in these cases which are used in Theorem $\ref{th:AsmpCov}$ below.
\end{remark}
The proof of the above lemma also gives asymptotic expressions for the functions $f_1, f_2$ and $f_3$, leading
to straightforward asymptotics for the expectation and covariance of $SFS$. 
The complete picture for the first moment is given in the next result. 

\begin{theorem}\label{cor:ASMEV_SFSI}
As $n$ goes to infinity,\\
(i) The expected number of external mutations ($b=1$) has the following asymptotics
\begin{equation*}\label{eq:EV_SFSlim1}
\frac{\log n}{n}\E[SFS_{n,1}]\to \theta.
\end{equation*}
(ii) If $b\geq2$ and $\frac{b}{n}\to 0$, then
\begin{equation*}\label{eq:EV_SFSlimb}
\frac{b(b-1)}{n}\log^2 \(\frac nb\)\E[SFS_{n,b}]\to \theta.
\end{equation*}
(iii) If $\frac{b}{n}\to u \in (0,1)$, then
\begin{equation*}\label{eq:EV_SFSlim}
n\E[SFS_{n,b}]\to \theta f_1(u)=\theta\int_0^1 u^{-1-p}(1-u)^{p-1} \frac{\sin(\pi p)}{\pi p}\dif{p}.
\end{equation*}
(iv) If $\frac{n-b}{n}\to 0$, then
\begin{equation*}
(n-b)\log \(\frac n{n-b}\) \E[SFS_{n,b}]\to \theta.
\end{equation*}
(v) Let $I=(x,y)$ with $0<x<y<1$ and define
\begin{equation*}\label{cvinterval}
SFS_{n,I}\coloneqq\sum_{b=\lceil nx \rceil }^{\lfloor ny \rfloor} SFS_{n,b}.
\end{equation*}
Then
\begin{equation*}\label{eq:EV_SFSIntervalLim}
\E\[SFS_{n,I}\]\to\E[SFS_I]
\end{equation*}
{as it is defined in \eqref{ESFSI}.}
\end{theorem}

Case (i) and case (ii) for fixed $b$ also follow from Theorem 4 in \cite{DK}. Cases (ii) and (iv) give an update to the approximation 
of the SFS for small and large families made in $\cite{NH13}$. \par

\begin{figure}
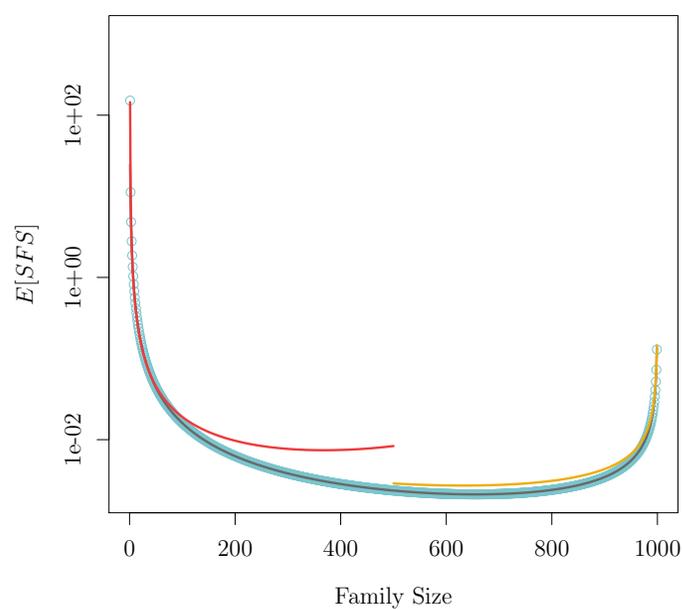

\centering
\scalebox{0.75}{

}
\caption{Exact and asymptotic approximations for $\E[SFS]$ in a population of size 1000: The blue circles give the exact value as given in Corollary $\ref{cor:SFS_EvCov}$. The gray line is the asymptotic approximation as given in Theorem $\ref{cor:ASMEV_SFSI}$ (iii). Red (resp. yellow) line is given by Theorem $\ref{cor:ASMEV_SFSI}$ (ii) (resp. (iv)).}
\label{fig:SFSAsmpCompare}
\end{figure}\par

In the same spirit and using the same techniques 
we now provide the complete picture for the second moments. In what follows we use the notation $f(n)\sim g(n)$ to denote that
\begin{equation*}
\frac{f(n)}{g(n)}\to 1
\end{equation*}
as $n\to\infty$.

\begin{theorem}\label{th:AsmpCov}
The covariance function has the following asymptotics as $n$ goes to infinity, in each of the following cases:\par
\hspace{-60pt}
\begin{tabular}{c c c c}
\\
\hline\hline
$b_1$ & $b_2-b_1$ & $n-b_2$ & $\Cov(SFS_{n,b_1},SFS_{n,b_2})$  \\ [0.5ex]
\hline
\\
$>1$ & $>0$ & $\sim n$ & 
$\frac{\theta^2}{b_1(b_1-1)b_2(b_2-1)} \bO{\frac{n^2}{\log^5n}}$ \\
$\sim n$ & $>0$ & $>0$ &
$\frac{\theta^2}{(b_2-b_1)(n-b_1)}\frac{1}{\log^2 n}$ \\
$\sim n$ & $0$ & $>0$ &
$\frac{\theta^2+\theta}{n-b_2}\frac{1}{\log n}$ \\
$>1$ & $\sim n$ & $=b_1$ &
$\theta^2 \bO{\frac{n}{\log^4 n}}$ \\
$>1$ & $\sim n$ & $=b_1+const^+$ &
$\theta^2L_1(n-b_2,b_1)\frac{n}{\log n}$ \\
$1$ & $0$ & $\sim n$&
$\theta^2 \bO{\frac{n^2}{\log^3 n}}$\\
$1$ & $>0$ & $\sim n$ &
$\theta^2 \bO{\frac{n^2}{\log^4 n}}$ \\
$1$ & $\sim nu$ & $\sim n(1-u)$ &
$\theta^2 \bO{\frac{1}{\log^2n}}$ \\
$1$ & $\sim n$ & $>1$ &
$\theta^2 \bO{\frac{n}{\log^3 n}}$ \\
$1$ & $\sim n$ & $1$ &
$\theta^2 \bO{\frac{n}{\log^3 n}}$ \\
$> 1$ & $0$ & $\sim n$ &
$\theta^2 \bO{\frac{n^2}{\log^5 n}}$\\
$\sim nu$ & $>0$ & $\sim n(1-u)$ &
$\frac{\theta^2}{(1-u)(b_2-b_1)}\frac{1}{n\log^2 n}$\\
$\sim nu$ & $0$ & $\sim n(1-u)$ &
$\frac{\theta f_1(u)}{n}$\\
$>1$ & $\sim nu$ & $\sim n(1-u)$ &
$\theta^2 \bO{\frac{1}{\log^3 n}}$\\
$\sim nu$ & $\sim n(1-u)$ & $>0$ &
$-\frac{\theta^2 f_1(u)}{n-b_2} \frac{1}{\log n} $\\
$\sim nu_1$ & $\sim nu_2$ & $\sim n(1-u_1-u_2)$ &
$\frac{\theta^2\(f_2(u_1,u_1+u_2)+f_3(u_1,u_1+u_2)\Ind{2u_1+u_2\leq 1} - f_1(u_1)f_1(u_1+u_2)\)}{n^2}$\\
$\sim nu$ & $\sim n(1-2u)$ & $=b_1$ &
$\frac{\theta^2\int_0^\infty\int_0^\infty \frac{e^{-y_1}e^{-y_2}}{y_1\vee y_2}\dif{y_1}\dif{y_2}}{u(1-u)}\frac{1}{n\log n}$\\
$\sim nu$ & $\sim n(1-2u)$ & $=b_1 + const^+$ &
$\frac{\theta^2\int_0^\infty\int_0^\infty \frac{e^{-y_1}e^{-y_2}(y_1+y_2)}{y_1\vee y_2}\dif{y_1}\dif{y_2}}{u(1-u)(n-b_2-b_1)}\frac{1}{n\log^2 n}$\\
\\
\hline
\\
\end{tabular}\par

Also for $I, \widehat I\subset(0,1)$, and $SFS_{n,{I}},SFS_{n,\widehat{I}}$ as defined in Theorem \ref{cvinterval} (V), we have
\begin{align}\label{eq:COV_SFSIntervalLim}
&\Cov\(SFS_{n,I},SFS_{n,\widehat{I}}\) \to \\
& \theta^2 \int_I \int_{\widehat{I}} f_2(u_1,u_2)+f_3(u_1,u_2)\Ind{u_1+u_2<1}
-f_1(u_1)f_1(u_2)
\dif{u_2}\dif{u_1}
+\theta
\int_{I\cap\widehat{I}}f_1(u)\dif{u}.\nonumber
\end{align}
\end{theorem}

These approximations follow from the asymptotics for $L_1,L_2,$ and $L_3$ substituted in the covariance formula given in Corollary $\ref{cor:SFS_EvCov}$. For the sake of simplicity we do not provide the explicit computations. We only treat the case where the expected value $\mathbb E[SFS_{n,b}]$ diverges, then an application of Chebyshev's inequality allows us to prove the following weak law of large numbers with $L^2$-convergence, which generalizes and strengthens results on the Bolthausen-Sznitman coalescent derived in $\cite{DK}$.

\begin{corollary}
\label{lawlargenumbers}
Suppose that $b/n\to 0$ in such a way that $\E[SFS_{n,b}]\to\infty$, or equivalently that $b=\lo{\sqrt{n}/\log n}$. Then we have the
following $L^2$-convergence:
\begin{equation*}
\frac{SFS_{n,b}}{\mathbb E[SFS_{n,b}]} \to \theta .
\end{equation*}
\end{corollary}

In view of Theorem \ref{cor:ASMEV_SFSI}
this means that
for $b=1$
\begin{equation*}
\frac{\log n}{n}\ SFS_{n,1} \to \theta,
\end{equation*}
and for $b \ge 2$, $b=\lo{\sqrt{n}/\log n}$
\begin{equation*}
\frac {b(b-1)\log^2\(n/b\)}{n}\ SFS_{n,b} \to \theta.
\end{equation*}


\section{Distribution of the Family-Sized Branch Lengths} \label{sec:BLdistribution}

In this section we 
discuss the particular case of $\ell_{n,b}$ when $b> n/2$.
In this case we are able to provide an explicit formula for the distribution function of the length of the coalescent of order $b$.
This leads to convergence in law results, but also to the law of $SFS_{n,b}$.
 Observe that in this case, for all $t\geq0$, $C_{n,b}(t)\in\{0,1\}$
and $\ell_{n,b}$ is just the time during which the block of size $b$ survives before 
coalescing with other blocks (if it ever exists, otherwise obviously $\ell_{n,b}=0$). 
We first find an expression for the distribution function of $\ell_{n,b}$.

\begin{theorem}\label{Prop:Llnb}
Suppose that ${n}/{2}< b<n$.
For any $s\geq0$,
\begin{equation}\label{eq:Llnb}
\P(\ell_{n,b}>s)
=\frac{n}{(n-b)b(b-1)}\int_0^1 \frac{\Psi(b-p)-\Psi(1-p)}{Be(n-b,e^{-s}p)Be(b-1,1-p)}
\dif{p}.
\end{equation}
\end{theorem}

From the derived distribution of $\ell_{n,b}$ in Theorem $\ref{Prop:Llnb}$ we obtain that, conditioned on $\ell_{n,b}>0$, the variable $\(\log n\)\ell_{n,b}$  has a limiting distribution.

\begin{corollary}\label{cor:ConvDistlnb}
Suppose that $b/n\to u\in[1/2,1)$ as $n\to\infty$, then letting $\alpha=\log(1-u)-\log u$, we have
\begin{equation*}
\frac{n}{\log n} \P(\ell_{n,b}>0) \to \frac{G(\alpha)}{u(1-u)}
\end{equation*}
where 
$$G(x)=\int_0^1 e^{px}
\frac{\sin \pi p}{\pi }\dif{p}=\frac{1+e^{x}}{\pi^2+x^2}.$$
Furthermore, 
\begin{equation*}
\P(\(\log n\)\ell_{n,b}>s \vert\ell_{n,b}>0) \to 
\frac{G(\alpha-s)}{G(\alpha)}.
\end{equation*}
\end{corollary}

We now give the joint distribution of the branch lengths for large families, i.e. the joint distribution of the vector 
$(\ell_{n,b})_{b>n/2}$. For this we introduce the following events: for any collection of integers $\mathbf{b}=(b_1,\cdots,b_m)$ such that 
$n/2< b_1<b_2<\cdots<b_m<n$, and any collection of nonnegative numbers $\mathbf{s}=(s_1,\cdots,s_m)$, define the event
\begin{equation*}
\Lambda_{\mathbf{b},\mathbf{s}}\coloneqq \(\bigcap_{i=1}^m \{\ell_{b_i}>s_i\}\) \bigcap 
\(\bigcap_{\substack{b>b_1\\b\not\in\mathbf{b}}}\{\ell_b=0\}\),
\end{equation*}
that is, the event that a block of size $b_1$ exists for a time larger than $s_1$, that this block then merges with some other blocks of total size exactly $b_2-b_1$, that this new block exists for a time larger than $s_2$, and so on, until the last merge of the growing block occurs with the remaining blocks of total size exactly $n-b_m$.

\begin{theorem}\label{le:JointLlnb}
For $\mathbf{b}=\(b_1,\cdots,b_m\)$ and $\mathbf{s}=(s_1,\cdots,s_m)$ as above, we have
\begin{equation}\label{eq:JointLlnb}
\P\(\Lambda_{\mathbf{b},\mathbf{s}}\)=\frac{n}{b_1(b_2-b_1)\cdots(n-b_m)}\frac{\exp\{-\IP{(m:1)}{\mathbf{s}}\}}{m!}
\int_0^1p^m\frac{\Psi(b_1-p)-\Psi(1-p)}{Be(b_1-1,1-p)}
\dif{p}
\end{equation}
and
\begin{align}\label{eq:JointLlnb2}
&\hspace{10pt}\P\(\Lambda_{\mathbf{b},\mathbf{s}},\bigcap_{n/2<b<b_1}\{\ell_{n,b}=0\}\)\\
&=\frac{n}{(b_2-b_1)\cdots(n-b_m)}\frac{\exp\{-\IP{(m:1)}{\mathbf{s}}\}}{m!}\times\nonumber\\
&\hspace{10pt}\(\int_0^1\frac{p^m}{b_1}\frac{\Psi(b_1-p)-\Psi(1-p)}{Be(b_1-1,1-p)}-
\frac{p^{m+1}}{m+1}\sum_{n/2<b<b_1} \frac{1}{b(b_1-b)}\frac{\Psi(b-p)-\Psi(1-p)}{Be(b-1,1-p)} \dif{p}\),\nonumber
\end{align}
where
\begin{equation*}
(m:1)\coloneqq(m,m-1,\dots,1).
\end{equation*}
and $\IP{\cdot}{\cdot}$ is the usual inner product in Euclidean space.
\end{theorem}

By conditioning on $(\ell_{n,b})_{b>n/2}$ and using equation $\eqref{eq:JointLlnb2}$ one can obtain a sampling formula for the vector 
$(SFS_{n,b})_{b>n/2}$, although the computations are rather convoluted and we do not present them here. 


\section{The approximations}\label{sec:approximations}

Here we derive the approximations given above in the Introduction. From Stirling's approximation we have the well-known formula $\Gamma(m+c)/\Gamma(m) \approx m^c$. Its application requires some care, since we shall apply this approximation also for small values of $m$ down to  $m=1$.  It is known and easily confirmed by computer that the approximation is particularly accurate within the range $0\le c \le 1$. Thus we use for $p\in(0,1)$ and $b \ge 2$ the approximations
\begin{equation*} \frac{\Gamma(b-p)}{\Gamma(b+1)}= \frac 1{b(b-1)}\frac{\Gamma (b-1+(1-p))}{\Gamma(b-1)} \approx \frac1{b(b-1)}(b-1)^{1-p} = \frac{(b-1)^{-p}} b
\end{equation*}
and
\begin{equation*} \frac{\Gamma(n-b+p)}{\Gamma(n-b+1)} = \frac 1{n-b} \frac{\Gamma(n-b+p)}{\Gamma(n-b)} \approx (n-b)^{p-1}.\end{equation*}
Also by Euler's reflection formula $\Gamma(1-p)\Gamma(1+p)= \pi p/\sin(\pi p)$.  Inserting these formulas into the expression \eqref{eq:Elnb} for the expected SFS we obtain
\begin{align*}
\mathbb E[SFS_{n,b}] &\approx \theta n \frac{b-1}b \int_0^1 (b-1)^{-p-1}(n-b)^{p-1} \frac{\sin (\pi p)}{\pi p} \dif p\\&= \theta \frac{n}{(n-1)^2} \frac {b-1}b f_1\Big(\frac{b-1}{n-1}\Big) .
\end{align*}
It turns out that this approximation overestimates the expected SFS, which can be somewhat counterbalanced by replacing the scaling factor $n/(n-1)^2$ by $1/(n-1)$. This  yields our first approximation \eqref{approx}. 

For the second approximation \eqref{approx2} we apply the expansion
\begin{equation*} \frac{\Gamma(m+c)}{\Gamma(m)} = m^c \Big(1- \frac{c(1-c)}{2m} + O(m^{-2})\Big), \end{equation*}
see \cite{TE51}. Again this approximation is particularly accurate for $0\le c \le 1$ leading for $p\in(0,1)$ and $b \ge 2$ to
\begin{align*}
\frac{\Gamma(b-p)}{\Gamma(b+1)}\frac{\Gamma(n-b+p)}{\Gamma(n-b+1)} &\approx \frac {(b-1)^{-p}}b(n-b)^{p-1} \Big(1- \frac{(1-p)p}{2(b-1)}\Big)\Big(1- \frac{p(1-p)}{2(n-b)}\Big)\\
&\approx  \frac {(b-1)^{-p}}b(n-b)^{p-1} \Big( 1- (n-1)\frac{p(1-p)}{2(b-1)(n-b)}\Big).
\end{align*} 
Using this approximation in the expression for the expected SFS we get for $b \ge 2$
\begin{align*}
\mathbb E[SFS_{n,b}] &\approx \theta n \frac{b-1}b\Bigg( \frac 1{(n-1)^2} f_1\Big(\frac {b-1}{n-1}\Big) \\& \hspace{2cm} \mbox{}- \frac{n-1}2 \int_0^1 (b-1)^{-p-2}(n-b)^{p-2} \frac{\sin (\pi p)}{\pi} (1-p)\, dp \Bigg)\\&=
\theta n \frac{b-1}b\Big( \frac 1{(n-1)^2} f_1\Big(\frac {b-1}{n-1}\Big) - \frac 1{(n-1)^3} g_1\Big(\frac{b-1}{n-1}\Big) \Big)
\end{align*}
with the function $g_1$
as defined in $\eqref{eq:g1}$.
This integral can be evaluated by elementary means yielding formula \eqref{approx2}.


\section{Proofs of Section $\ref{sec:BLmoments}$}\label{sec:proofsBLmoments}
As in the infinite coalescent case, the proof of Theorem $\ref{th:Elnb}$ begins with the definition \eqref{defellCn}  and by noting that
\begin{equation*}
\E\[\ell_{n,b}\]=\E\[\int_0^\infty C_{n.b}(t) \dif{t}\]=\int_0^\infty \E\[C_{n,b}(t)\]\dif{t},
\end{equation*}
and similarly
\begin{equation*}
\E\[\ell_{n,b_1}\ell_{n,b_2}\]=\int_0^\infty\int_0^\infty \E\[C_{n,b_1}(t_1)C_{n,b_2}(t_2)\]\dif{t_1}\dif{t_2},
\end{equation*}
so it only remains to compute {$\E\[C_{n,b}(t)\]$} and $E\[C_{n,b_1}(t)C_{n,b_2}(t)\]$ in each case and simplify the expressions.

\begin{proof}[Proof of Theorem $\ref{th:Elnb}$ (first moment)]
Let $\mathcal{B}$ be the collection of all possible blocks of size $b$ in a partition of $[n]$. Then
\begin{equation*}
\E\[C_{n,b}(t)\]=\E\[\sum_{B\in\mathcal{B}} \Ind{B\in\Pi^n(t)}\]=\sum_{B\in\mathcal{B}}\P\(B\in\Pi^n(t)\),
\end{equation*}
and by exchangeability of $\Pi^n(t)$,
\begin{equation*}
\E\[C_{n,b}(t)\]=\binom{n}{b}\P\(\{1,\cdots,b\}\in\Pi^n(t)\).
\end{equation*}
Thus, using  \eqref{Le:EVPD}, the fact that $\DAF{\Pi^\infty(t)}=:(A_1, A_2,\dots)\overset{d}{=}PD(e^{-t},0)$,
and writing $\Pi^n$ as $\Pi^\infty_{|n}$, we obtain
\begin{align*}
\E[C_{n,b}(t)]&=\binom{n}{b}\E\[\sum_{i=1}^\infty A_i^{b}(1-A_i)^{n-b}\]\\
&=\binom{n}{b}\int_0^1 u^{b-1}(1-u)^{n-b} \frac{u^{-e^{-t}}(1-u)^{e^{-t}-1}}{\Gamma(1-e^{-t})\Gamma(e^{-t})} \dif u\\
&=\frac{n\Gamma(n)}{\Gamma(n-b+1)\Gamma(b+1)}\frac{Be(b-e^{-t},n-b+e^{-t})}{\Gamma(1-e^{-t})\Gamma(1+e^{-t})}
.
\end{align*}
Finally, by changing the variable $p=e^{-t}$, we obtain \eqref{eq:Elnb}.
\end{proof}

Now we use the random tree construction of the $n$-Bolthausen-Sznitman coalescent in order to compute the second moments
of $\ell_{n,b}$.
\begin{proof}[Proof of Theorem $\ref{th:Elnb}$ (second moments)]
Let $1\leq b_1 \leq b_2 \leq n-1$, and $\mathcal{B}_1,\mathcal{B}_2$ be the collection
of all possible blocks of sizes $b_1$ and $b_2$ respectively in a partition of $[n]$.
Then
\begin{align}\label{eq:Elb1lb2_3}
\E\[\ell_{n,b_1}\ell_{n,b_2}\]
=&\int_0^\infty \int_0^\infty \E\[ C_{n,b_1}(t_1) C_{n,b_2}(t_2) \] \dif{t_2}\dif{t_1} \nonumber\\
=&\int_0^\infty \int_0^\infty \sum_{B_1\in\mathcal{B}_1}\sum_{B_2\in\mathcal{B}_2}  
         \P\(B_1\in\Pi^n(t_1),B_2\in\Pi^n(t_2)\)
  \dif{t_2}\dif{t_1}.
\end{align}
We now compute $\P\(B_1\in\Pi^n(t_1),B_2\in\Pi^n(t_2)\)$ by cases. \par
i) Suppose that {$B_1\cap B_2=\emptyset$}.
By exchangeability we have
\begin{equation*}
\P\(B_1\in\Pi^n(t_1),B_2\in\Pi^n(t_2)\)=\P(\{1,\cdots,b_1\}\in\Pi^n(t_1),
                                                 \{b_1+1,\cdots,b_1+b_2\}\in\Pi^n(t_2))
\end{equation*}
where this probability is of course 0 if $b_1+b_2>n$.
Now suppose that $t_1\leq t_2$.
In terms of the RRT construction of the Bolthausen-Sznitman coalescent, the event 
\begin{equation*}
\{\{1,\cdots,b_1\}\in\Pi^n(t_1),\{b_1+1,\cdots,b_1+b_2\}\in\Pi^n(t_2)\}
\end{equation*}

is characterized by a RRT with exponential edges, say $E_2,\cdots,E_{n}$, constructed as follows: 
for $i\in\{1,\cdots,b_1-1\}$ the
node $\{i+1\}$ along with $E_{i+1}$ arrive to the tree but with the imposed restriction that it may 
not attach to $\{1\}$ and
have $E_{i+1}>t_1$ at the same time, 
which occurs with probability $e^{-t_1}/i$; this ensures
that $\{i+1\}$ coalesces with $\{1\}$ before time $t_1$ for all $i< b_1$, 
thus creating the block $\{1,\cdots,b_1\}$ up to
time $t_1$. After $\{1\},\cdots, \{b_1\}$ have arrived, 
the node $\{b_1+1\}$ must attach 
to $\{1\}$ and $E_{b_1+1}$ must be greater than $t_2$, which occurs with probability $e^{-t_2}/b_1$; the node
$\{b_1+1\}$ will be the root of a sub-tree formed with the nodes $\{b_1+2\},\cdots,\{b_1+b_2\}$ which
will build the block $\{b_1+1,\cdots,b_1+b_2\}$ at time $t_2$. Thus,
for each
$i\in\{1,\cdots,b_2-1\}$ the node $\{b_1+i+1\}$ must arrive and attach to any of $\{b_1+1\},\cdots,\{b_1+i\},$
which occurs with probability $\frac{i}{b_1+i}$, and, furthermore,
conditional on this event, 
it may not
attach to $\{b_1+1\}$ and have $E_{b_1+i+1}>t_2$ at the same time, which occurs with probability 
$\frac{e^{-t_2}}i$. 
Finally, if $n-b_1-b_2>0$, for $i\in\{0,\cdots,n-b_1-b_2-1\}$ the node $\{b_1+b_2+i+1\}$ must either attach to any of
$\{b_1+b_2+j\}$, $1\leq j \leq i$, or attach to $\{1\}$ or $\{b_1+1\}$ and have
$E_{b_1+b_2+i+1}>t_1$ or $E_{b_1+b_2+i+1}>t_2$ respectively; 
this occurs with probability $\frac{e^{-t_1}+e^{-t_2}+i}{b_1+b_2+i}$.  
Putting all together we obtain
\begin{align*}\hspace{-20pt}
&\P\(B_1\in\Pi^n(t_1),B_2\in\Pi^n(t_2)\)\\
=&\[\prod_{i=1}^{b_1-1}\(1-\frac{e^{-t_1}}{i}\)\]  
  \[\frac{e^{-t_2}}{b_1}\prod_{i=1}^{b_2-1}\(1-\frac{e^{-t_2}}{i}\)\frac{i}{b_1+i}\]
  \[\prod_{i=0}^{n-b_1-b_2-1} \frac{e^{-t_1}+e^{-t_2}+i}{b_1+b_2+i}\]\\
=&\frac{1}{(n-1)!}\frac{\Gamma(b_1-e^{-t_1})}{\Gamma(1-e^{-t_1})}
  e^{-t_2}
  \frac{\Gamma(b_2-e^{-t_2})}{\Gamma(1-e^{-t_2})}
  \frac{\Gamma(n-b_1-b_2+e^{-t_1}+e^{-t_2})}{\Gamma(e^{-t_1}+e^{-t_2})},
\end{align*}
where the last product is set to 1 if $n-b_2-b_1=0$.
On the other hand,
if $t_2<t_1$, by exchangeability we may instead compute 
\begin{equation*}
\P(\{1,\cdots,b_2\}\in\Pi^n(t_2),\{b_2+1,\cdots,b_2+b_1\}\in\Pi^n(t_1))
\end{equation*}
obtaining
\begin{align*}\hspace{-20pt}
&\P\(B_1\in\Pi^n(t_1),B_2\in\Pi^n(t_2)\)\\=&
\frac{1}{(n-1)!}\frac{\Gamma(b_2-e^{-t_2})}{\Gamma(1-e^{-t_2})}
e^{-t_1}
\frac{\Gamma(b_1-e^{-t_1})}{\Gamma(1-e^{-t_1})}
\frac{\Gamma(n-b_2-b_1+e^{-t_2}+e^{-t_1})}{\Gamma(e^{-t_2}+e^{-t_1})}.
\end{align*}\par
ii) Suppose that{ $B_1\subset B_2$}.
Of course if $t_1>t_2$ we have
$\P\(B_1\in\Pi^n(t_1),B_2\in\Pi^n(t_2)\)=0$ whenever $B_1$ is strictly contained in $B_2$. Assuming that $t_1\leq t_2$ and using the same rationale as before we obtain
\begin{align*}
&\P\(B_1\in\Pi^n(t_1),B_2\in\Pi^n(t_2)\)\\
=&\[\prod_{i=1}^{b_1-1}\frac{i-e^{-t_1}}{i}\]
 \[\prod_{i=0}^{b_2-b_1-1}\frac{i+e^{-t_1}-e^{-t_2}}{b_1+i}\]
 \[\prod_{i=0}^{n-b_2-1}\frac{e^{-t_2}+i}{b_2+i}\]\\
=&\frac{1}{(n-1)!}
  \frac{\Gamma(b_1-e^{-t_1})}{\Gamma(1-e^{-t_1})}
  \frac{\Gamma(b_2-b_1+e^{-t_1}-e^{-t_2})}{\Gamma(e^{-t_1}-e^{-t_2})}
  \frac{\Gamma(n-b_2+e^{-t_2})}{\Gamma(e^{-t_2})},
\end{align*}
where the product in the middle is set to $1$ if $B_1=B_2$.\par
iii) If {$B_1\cap B_2\neq \emptyset$ and $B_1\not\subset B_2$}, 
we clearly have 
$\P\(B_1\in\Pi^n(t_1),B_2\in\Pi^n(t_2)\)=0$. \par
From the previous computations, and summing over the corresponding cases, we see that 
if $b_1+b_2\leq n$ then, changing the variable $p=e^{-t}$, the integral in $\eqref{eq:Elb1lb2_3}$ is given by
\begin{align*}
\E\[\ell_{n,b_1}\ell_{n,b_2}\]
=&\frac{n}{b_1!b_2!(n-b_1-b_2)!}\\
&\int_0^1\int_0^1
\frac{\Gamma(b_1-p_1)}{\Gamma(1-p_1)}
\frac{\Gamma(b_2-p_2)}{\Gamma(1-p_2)}
\frac{\Gamma(n-b_1-b_2+p_1+ p_2)}{\Gamma(p_1 + p_2)}
\frac{\dif{p_1}\dif{p_2}}{p_1\vee p_2}\\+
&\frac{n}{b_1!(b_2-b_1)!(n-b_2)!}\\
&\int_0^1\int_{0}^{p_1}
\frac{\Gamma(b_1-p_1)}{\Gamma(1-p_1)}
 \frac{\Gamma(b_2-b_1+p_1-p_2)}{\Gamma(p_1-p_2)}
 \frac{\Gamma(n-b_2+p_2)}{\Gamma(p_2+1)}
\frac{\dif{p_2}\dif{p_1}}{p_1}
\end{align*}
whereas if $b_1+b_2>n$ the first summand in the above expression is set to zero. Rearranging terms
we obtain \eqref{eq:Elb1lb2}. 
\end{proof}

\begin{proof}[Proof of Lemma $\ref{le:ASMPL1L2L3}$ (asymptotics for $L_1$)]
Again, we have from Stirling's formula that $\Gamma(m+c)/\Gamma(m+d)= m^{c-d}(1+\bO{1/m})$ 
for any real numbers $c$ and $d$, where the $\bO{1/m}$ term holds uniformly for $0\leq c,d\leq 1$. 
Letting $m=b-1$ and $n-b$ leads to the following equality:
\begin{align*}\frac{n}{b(n-b)}&\frac{\Gamma(n-b+p)}{\Gamma(n-b)}
\frac{\Gamma(b-p)}{\Gamma(b)}\\&=\frac{n}{b(n-b)}(n-b)^{p}(b-1)^{-p}\(1+\bO{\frac{1}{b}}+\bO{\frac{1}{n-b}}\).
\end{align*}
Thus, using Euler's reflection formula to write $\Gamma(1-p)\Gamma(1+p)$ as $\pi p/\sin{\(\pi p\)}$ in the definition of $L_1$, we get
\begin{align*}
L_1(n,b)=&\(1+\bO{\frac{1}{b}}+\bO{\frac{1}{n-b}}\)\frac{1}{b(n-b)}
\int_0^1 \frac{\sin{\(\pi p\)}}{\pi p} \(\frac{n-b}{b-1}\)^p\dif{p}\\
=&\(1+\bO{\frac{1}{b}}+\bO{\frac{1}{n-b}}\)\frac{b-1}{b (n-1)^2}f_1\(\frac{b-1}{n-1}\)
\end{align*}
Thus, for every $\epsilon>0$ there is a $b_0\in\N$ such that for large enough $n\in\N$ we have
\begin{equation}\label{eq:maxAsmpL1}
\max_{b_0\leq b \leq n-b_0}\abs{\frac{n^2L_1(n,b)}{f_1\(\frac{b-1}{n-1}\)} -  \frac{b-1}{b}}< \epsilon.
\end{equation}
It remains to study the approximation as $n\to\infty$ in the cases where $n-b$ or $b$ remain constant. In the first
case, when $n-b=c$, we have $b\to\infty$ as $n\to\infty$ and, by Stirling's approximation and dominated convergence and substituting 
$p = y/ \log b$ on the one hand

\begin{align*}\hspace{-30pt}
L_1(n,b)
\sim&
\int_0^1 \frac{\sin{\(\pi p\)}}{\pi p} b^{-p-1} \frac{\Gamma(c+p)}{\Gamma(c+1)}\dif{p}\\
=& \frac{1}{bc}\int_0^{\log b} \frac{\sin{\(\pi y / \log b\)}}{\pi y /\log b} e^{-y} 
\frac{\Gamma(c+y/\log b)}{\Gamma(c)}\frac{\dif{y}}{\log b}\\
\sim&\frac{1}{bc\log b} \int_0^\infty e^{-y} \dif{y}.
\end{align*}
and on the other hand because of $b \to \infty$
\begin{align*}
\frac{1}{n^2}f_1\(\frac{b-1}{n-1}\) \sim& \frac{1}{bc} \int_0^1 \frac{\sin(\pi p)}{\pi p} b^{-p} c^p \dif{p}\\
=& \frac{1}{bc} \int_0^{\log b} \frac{\sin\(\pi y/\log b\)}{\pi y/\log b} e^{-y} c^{y/\log b} \frac{\dif{y}}{\log b} \\
\sim& \frac{1}{bc\log b} \int_0^\infty e^{-y} \dif{y}.
\end{align*}
Thus $L_1(n,b)\sim n^{-2}f_1((b-1)/(n-1))$ which extends $\eqref{eq:maxAsmpL1}$ for $b>n-b_0$.\par

Similarly for the second case, if $b\geq 2$ is fixed, we have $n-b\to\infty$ as $n\to\infty$. Thus,
 with $1-p=y/\log n$
\begin{align*}
L_1(n,b)
\sim&
\int_0^1 \frac{\sin{\(\pi p\)}}{\pi p} \frac{\Gamma(b-p)}{\Gamma(b+1)} n^{p-1} \dif{p}\\
=&
\frac{1}{\log^2(n)}\int_0^{\log n} 
\frac{\sin\(\pi - \pi y /\log n\)}{(1-y/\log n)\pi y / \log n }
\frac{\Gamma\(b-1+\frac{y}{\log n}\)}{\Gamma(b+1)}ye^{-y}\dif{y} \\
\sim&
\frac{1}{b(b-1)\log^2 n} \int_0^\infty ye^{-y} \dif{y}
\end{align*}
and
\begin{align}
\frac{1}{n^2}f_1&\(\frac{b-1}{n-1}\) \sim \frac{1}{(b-1)^2} \int_0^1 \frac{\sin \pi p}{\pi p} (b-1)^{1-p} n^{p-1}\dif{p} \nonumber\\
=& \frac{1}{(b-1)^2\log^2 n} \int_0^{\log n} \frac{\sin\(\pi - \pi y / \log n\)}{\(1-y/\log n\)\pi y /\log n}
(b-1)^{y/\log n} ye^{-y} \dif{y} \nonumber\\
\sim& \frac{1}{(b-1)^2\log^2 n}  \int_0^\infty  ye^{-y}\dif{y}. \label{eq:PASMPL1bfixed}
\end{align}
Thus $L_1(n,b)\sim (b-1)n^{-2}f_1((b-1)/(n-1))/b$, which extends $\eqref{eq:maxAsmpL1}$ for $b<b_0$. This  extends $\eqref{eq:maxAsmpL1}$ for $b<b_0$. Thus we proved $\eqref{eq:ASMPL1}$.

For the proof of $\eqref{eq:ASM_EV_SFS1}$, we substitute
$b$ by 1 and perform similar computations:
\begin{align*}
L_1(n,1)&= \int_0^1  \frac{\Gamma(1-p)}{\Gamma(2)} \frac{\Gamma(n-1+p)}{\Gamma(n)} \frac{\dif{p}}{\Gamma(1-p)\Gamma(1+p)}\\
&\sim \int_0^1 n^{p-1} \frac{\dif{p}}{\Gamma(1+p)}\\
&= \int_0^{\log n} e^{-y} \frac{\dif{y}}{(\log n) \Gamma(2-y/\log n)}\\
&\sim\frac{1}{\log n}\int_0^\infty e^{-y} \dif{y},
\end{align*}
and from $\eqref{eq:PASMPL1bfixed}$ with choosing $b=2$
\begin{equation*}
\frac{1}{n^2}f_1\(\frac{1}{n-1}\) \sim \frac{1}{\log^2 n } \int_0^\infty ye^{-y} \dif{y}.
\end{equation*}
This proves $\eqref{eq:ASM_EV_SFS1}$.
\end{proof}

\begin{proof}[Proof of Lemma $\ref{le:ASMPL1L2L3}$ (asymptotics for $L_2$ and $L_3$) ]
The arguments here are similar to the arguments in the proof of the asymptotics for $L_1$, 
but we avoid repeating similar and tedious computations. We only layout the first steps of the proof.
By Stirling's approximation applied to the integrands appearing in
$L_2$ and $L_3$, we obtain, for $b_2-b_1>0$,
\begin{align*}
&\frac{\Gamma(b_1-p_1)}{\Gamma(b_1+1)}
\frac{\Gamma(b_2-b_1+p_1-p_2)}{\Gamma(b_2-b_1+1)}
\frac{\Gamma(n-b_2+p_2)}{\Gamma(n-b_2+1)}=\\
&\frac{1}{(n-1)^3}\(\frac{b_1-1}{n-1}\)^{-p_1-1}
\(\frac{b_2-b_1}{n-1}\)^{p_1-p_2-1}
\(\frac{n-b_2}{n-1}\)^{p_2-1}\times \\
&\(1+\bO{\frac{1}{b_1}}+\bO{\frac{1}{b_2-b_1}}+\bO{\frac{1}{n-b_2}}\),
\end{align*}
and, for $n-b_2-b_1>0$,
\begin{align*}
\hspace{-20in}
&\frac{\Gamma(b_1-p_1)}{\Gamma(b_1+1)}
 \frac{\Gamma(b_2-p_2)}{\Gamma(b_2+1)}
 \frac{\Gamma(n-b_1-b_2+p_1+p_2)}{\Gamma(n-b_1-b_2+1)}=\\
&\frac{1}{(n-2)^3}\(\frac{b_1-1}{n-2}\)^{-p_1-1}
\(\frac{b_2-1}{n-2}\)^{-p_2-1}\(1-\frac{b_1+b_2}{n-2}\)^{p_1+p_2-1} \times \\
&\(1+\bO{\frac{1}{b_1}}+\bO{\frac{1}{b_2}}+\bO{\frac{1}{n-b_1-b_2}}\);
\end{align*}
thus
\begin{equation*}
L_2(n,b_1,b_2)=\frac{1}{(n-1)^3}f_2\(\frac{b_1-1}{n-1},\frac{b_2-1}{n-1}\)\(1+\bO{\frac{1}{b_1}}+\bO{\frac{1}{b_2-b_1}}+
\bO{\frac{1}{n-b_2}}\),
\end{equation*}
and
\begin{equation*}
L_3(n,b_1,b_2)=\frac{1}{(n-2)^3}f_3\(\frac{b_1-1}{n-2},\frac{b_2-1}{n-2}\)\(1+\bO{\frac{1}{b_1}}+\bO{\frac{1}{b_2}}+
\bO{\frac{1}{n-b_1-b_2}}\).
\end{equation*}
Similar to the analysis in the proof of $\eqref{eq:ASMPL1}$, to obtain \eqref{eq:ASMPL2} 
it remains to study the cases where at least one
of $b_1,b_2-b_1,$ or $n-b_2$ remains constant, whereas for $\eqref{eq:ASMPL3}$ the cases of interest are
where one of $b_1,b_2,$ or $n-b_2-b_1$ remain constant.  
\end{proof}
\begin{proof}[Proof of Theorem \ref{cor:ASMEV_SFSI}]
We first derive the asymptotic  behavior of the function $f_1$.  We have 
\begin{align}
f_1(u) \sim \frac 1{u^2\log^2 u}  \quad \text{ as } u\downarrow 0.
\label{f1at0}
\end{align}
For the proof note that for $u<1/2$ we have $(1-u)^{p-1} \le 2$. Therefore dominated convergence implies for $u\downarrow 0$  
\begin{align*}
f_1(u) &= \frac 1{u^2}\int_0^1 u^{1-p}(1-u)^{p-1} \frac { \dif p}{\Gamma(1-p)\Gamma(1+p)}
\\&= \frac 1{u^2} \int_0^1 e^{-(p-1)\log u }(1-u)^{p-1}(1-p) \frac { \dif p}{\Gamma(2-p)\Gamma(1+p)}
\\& = \frac 1{u^2}\int_0^{-\log u} e^{-y} (1-u)^{y/\log \frac 1u} \frac y{\log \frac 1u} \cdot\frac {\dif y}{\log\frac 1 u \Gamma(1-\frac y{\log u})\Gamma(2+\frac y{\log u})}
\\& \sim \frac 1{u^2\log^2 u} \int_0^\infty ye^{-y} \dif y 
\end{align*}
implying \eqref{f1at0}. Also
\begin{align}
f_1(u) \sim -\frac 1{(1-u)\log (1-u)}  \quad \text{ as } u\uparrow 1,
\label{f1at1}
\end{align}
which we obtain again by means of dominated convergence in the limit $u \uparrow 1$ as follows:
\begin{align*}
f_1(u) &=\frac 1{u(1-u)} \int_0^1 e^{p\log (1-u)} u^{-p} \frac { \dif p}{\Gamma(1-p)\Gamma(1+p)}
\\& =\frac 1{u(1-u)} \int_0^{-\log(1-u)} \frac{e^{-y} u^{y/\log(1-u)} \dif y}{(-\log(1-u))\Gamma(1+\frac y{\log(1-u)})\Gamma(1-\frac y{\log(1-u)})}
\\& \sim- \frac 1{(1-u)\log(1-u) }\int_0^\infty e^{-y} \dif y.
\end{align*}

These asymptotics together with Lemma \ref{le:ASMPL1L2L3} imply our claims. Without loss of generality let $\theta=1$. From \eqref{eq:ASM_EV_SFS1} we obtain
\begin{equation*}
\E\[ SFS_{n,1}\]=  nL_1(n,1) \sim \frac{1}n f_1\(\frac 1{n-1}\) \sim \frac{\log n}n \frac{(n-1)^2}{\log^2(n-1)} 
\end{equation*}
which yields claim (i). 

Similary from \eqref{eq:ASMPL1} we get for $b\ge 2$ and $b/n\to 0$
\begin{equation*}
\E\[ SFS_{n,b}\]=  nL_1(n,b) \sim \frac{b-1}{nb}   f_1\(\frac {b-1}{n-1}\) \sim \frac{b-1}{nb}  \frac{(n-1)^2}{(b-1)^2\log^2\frac{b-1}{n-1}}
\end{equation*}
which in view of $b/n \to 0$ yields assertion (ii).

Claim (iii) is an immediate consequence of formula \eqref{eq:ASMPL1}, since here we have $(b-1)/b\to 1$.

Next under the condition $(n-b)/n\to 0$ we get from \eqref{eq:ASMPL1} and \eqref{f1at1}
\begin{align*}
\E\[ SFS_{n,b}\] \sim \frac{b-1}{nb}   f_1\(\frac {b-1}{n-1}\) \sim- \frac {b-1}{nb} \frac {n-1}{(n-b)\log \frac{n-b}{n-1}}\sim \frac 1{(n-b) \log \frac n{n-b}}
\end{align*}
which confirms assertion (iv).

Finally, we have from \eqref{eq:ASMPL1}
\begin{equation*} 
\E\[ SFS_{n,I}\] \sim \frac{1}{n} \sum_{\frac bn \in I} f_1\(\frac bn \) \sim \int_I f_1(u) \dif u,
\end{equation*}
which is claim (v). This finishes the proof.
\end{proof}

\begin{proof}[Proof of Theorem $\ref{th:AsmpCov}$]
The approximations follow from the asymptotics for $L_1,L_2,$ and $L_3$ substituted in the covariance formula given in Corollary $\ref{cor:SFS_EvCov}$.
\end{proof}

\begin{proof}[Proof of Corollary \ref{lawlargenumbers}]
We have to prove that
\begin{equation*} \mathbb Var (SFS_{n,b})= o\big( \mathbb E[SFS_{n,b}]^2\big).
\end{equation*}
From the monotonicity properties of the gamma function
we have for $1\le b \leq n-1$
\begin{align}
L_2(n,b,b) &= \int_0^1\int_0^{p_1} \frac{\Gamma(b-p_1)}{\Gamma(b+1)}\frac{\Gamma(n-b+p_2)}{\Gamma(n-b+1)} \frac{\dif p_2 \dif p_1}{p_1 \Gamma(1-p_1)\Gamma(p_2+1)}\notag \\
&\leq \int_0^1 \frac{\Gamma(b-p_1)}{\Gamma(b+1)}\frac{\Gamma(n-b+p_1)}{\Gamma(n-b+1)} \frac{1}{\Gamma(1-p_1)p_1}
      \int_0^{p_1} \frac{\Gamma(1+p_1)}{\Gamma(1+p_1)\Gamma(1+p_2)} \dif p_2 \dif p_1\notag \\
&\le \sup_{1\le x\le y  \le 2} \frac {\Gamma(y)}{\Gamma(x)} \int_0^1 \frac{\Gamma(b-p_1)}{\Gamma(b+1)}\frac{\Gamma(n-b+p_1)}{\Gamma(n-b+1)} \frac{ \dif p_1}{ \Gamma(1-p_1)\Gamma(p_1+1)}\notag \\
&=\sup_{1\le x\le y  \le 2} \frac {\Gamma(y)}{\Gamma(x)} L_1(n,b) .
\label{lawlargenumbers1}
\end{align}
Concerning $L_3(n,b,b)$ we have for $b=o(n)$ by Stirling's approximation uniformly in $0 \le p_1 , p_2 \le 1$ 
\begin{align*}
\frac{\Gamma(n-2b+p_1+p_2)}{\Gamma(n-2b+1)} \sim n\frac{\Gamma(n-b+p_1)}{\Gamma (n-b+1)}\frac{\Gamma(n-b+p_2)}{\Gamma (n-b+1)},
\end{align*}
hence,  
with $1<\eta <2$
\begin{align}
\iint\limits_{\overset{\scriptstyle 0\le p_1,p_2\le   1}{ \eta< p_1+p_2 \le 2}} &\frac{\Gamma(b-p_1)}{\Gamma(b+1)} \frac{\Gamma(b-p_2)}{\Gamma(b+1)}\frac{\Gamma(n-2b+p_1+p_2)}{\Gamma(n-2b+1)}\notag \\
&\qquad \times \frac{\dif p_2 \dif  p_1}{\Gamma(1-p_1)\Gamma(1-p_2) (p_1\vee p_2)\Gamma (p_1+p_2)}\notag
\\&\sim  n \iint\limits_{\overset{\scriptstyle 0\le p_1,p_2\le   1}{ \eta< p_1+p_2 \le 2}}^{{\color{white} |}}  \frac{\Gamma(b-p_1)}{\Gamma(b+1)} \frac{\Gamma(b-p_2)}{\Gamma(b+1)}\frac{\Gamma(n-b+p_1)}{\Gamma (n-b+1)}\frac{\Gamma(n-b+p_2)}{\Gamma (n-b+1)}\notag
\\ &\qquad \times \frac{\dif p_2 \dif  p_1}{\Gamma(1-p_1)\Gamma(1-p_2) (p_1\vee p_2)\Gamma (p_1+p_2)}\notag
\\ &\le \frac n{\eta-1}\sup_{\eta \le x \le 2}\frac 1{\Gamma (x)} \int_0^1\int_0^{1{\color{white} \big|}} \frac{\Gamma(b-p_1)}{\Gamma(b+1)} \frac{\Gamma(b-p_2)}{\Gamma(b+1)}\frac{\Gamma(n-b+p_1)}{\Gamma (n-b+1)}\notag
\\ &\qquad \times \frac{\Gamma(n-b+p_2)}{\Gamma (n-b+1)}\frac{\dif p_2 \dif  p_1}{\Gamma(1-p_1)\Gamma(1-p_2) \Gamma(1+p_1)\Gamma(1+p_2)} \notag
\\&= \frac n{\eta-1}\sup_{\eta \le x \le 2}\frac 1{\Gamma (x)} L_1(n,b)^2 .
\label{lawlargenumbers2}
\end{align}
Also, by another application of Stirling's approximation and for $b =o(n)$
\begin{align}
\iint\limits_{\overset{\scriptstyle 0\le p_1,p_2\le   1}{ 0< p_1+p_2 \le \eta}} &\frac{\Gamma(b-p_1)}{\Gamma(b+1)} \frac{\Gamma(b-p_2)}{\Gamma(b+1)}\frac{\Gamma(n-2b+p_1+p_2)}{\Gamma(n-2b+1)}\notag\\
&\qquad \times \frac{\dif p_2 \dif  p_1}{\Gamma(1-p_1)\Gamma(1-p_2) (p_1\vee p_2)\Gamma (p_1+p_2)} \notag
\\& =O\Big( \iint\limits_{\overset{\scriptstyle 0\le p_1,p_2\le   1}{ 0< p_1+p_2 \le \eta}} b^{-p_1-p_2-2}(n-2b)^{p_1+p_2-1}\notag
\\
&\qquad \times \frac{\dif p_2 \dif  p_1}{\Gamma(1-p_1)\Gamma(1-p_2) (p_1\vee p_2)\Gamma (p_1+p_2)}\Big)\notag
\\& =O\Big( b^{-\eta-2}(n-2b)^{\eta-1}\iint\limits_{0\le p_1p_2 \le 1}^{{\color{white}|}} \frac{\dif p_2 \dif  p_1}{\Gamma(1-p_1)\Gamma(1-p_2) (p_1\vee p_2)\Gamma (p_1+p_2)}\Big)\notag
\\& = o\Big( \frac n{b^4 \log^4 n}\Big)
\label{lawlargenumbers3}
\end{align}
Combining \eqref{lawlargenumbers2} and \eqref{lawlargenumbers3} with Theorem \ref{cor:ASMEV_SFSI} (i) and (ii) and letting $\eta \to 2$ we obtain
\begin{equation*} L_3(n,b,b) = n L_1(n,b)^2(1+o(1)) + o( n^{-1} \mathbb E[SFS_{n,b}]^2) .
\end{equation*}
Using this estimate together with \eqref{lawlargenumbers1} and with Theorem \ref{th:Elnb}, Corollary \ref{cor:SFS_EvCov} yields
\begin{equation*}
\mathbb Var (SFS_{n,b}) = O( \mathbb E[SFS_{n,b}])+ o(\mathbb E[SFS_{n,b}]^2 )
\end{equation*}
Because of our assumption $\mathbb E[SFS_{n,b}]\to \infty$ our claim is proved.
\end{proof}


\section{Proofs of Section \ref{sec:BLdistribution}} \label{sec:proofsBLdistribution}
\begin{proof}[Proof of Theorem $\ref{Prop:Llnb}$]
Note that since $b> n/2 $, and by the exchangeability of $\Pi^n$, we have:
\begin{equation*}
\P(\ell_{n,b}>s)=\binom{n}{b} \P\(\L{}(\{t:\{1,\cdots,b\}\in\Pi^n(t)\})>s\),
\end{equation*}
where $\L{}$ is the Lebesgue measure, and $\L{}(\{t:\{1,\cdots,b\}\in\Pi^n(t)\})$ gives the time
that the block $\{1,\cdots,b\}$ exists in the Bolthausen-Sznitman coalescent starting with $n$ individuals.\par

We now describe the event $\{\L{}(\{t:\{1,\cdots,b\}\}\in\Pi^n(t))>s\}$ 
in terms of the RRT construction of the Bolthausen-Sznitman coalescent. Let $\mathcal{G}$ be the event that
the nodes $\{1\},\{2\},\cdots,\{b\}$ and $\{1\},\{b+1\},\cdots,\{n\}$ form two sub-trees,
say $T_1$ and $T_2$ rooted at $\{1\}$; i.e.
\begin{align*}
\mathcal{G}\coloneqq&\{T\colon \{j\} \text{ does not attach to } \{i\} \text{, for all } 2\leq i \leq b \text{ and } b < j \leq n\}.
\end{align*}
Then
\begin{equation*}
\L{}(\{t:\{1,\cdots,b\}\}\in\Pi^n(t))=
\begin{cases*}
0 & if $T\not\in \mathcal{G}$ \\
\(m\(T_2\)-M\(T_1\)\)\vee0 & if $T \in \mathcal{G}$.
\end{cases*}
\end{equation*}
Indeed, observe that by the cutting-merge procedure $T\not\in\mathcal{G}$ if and only if any block of 
$\Pi^n$ that contains all of $\{1,\cdots,b\}$ also contains some $j\in \{b+1,\cdots, n\}$.
 On the other hand, on the event
$\{T\in \mathcal{G}\}$, the random variable $M(T_1)$ is just the time at which the block $\{1,\cdots,b\}$ 
appears in $\Pi^n$, while
$m(T_2)$ is the time at which it coalesces with some other block in $T_2$.
Furthermore, observe that conditioned on $\{T\in\mathcal{G}\}$, $T_1$ and $T_2$ are two independent
RRTs of sizes $b$ and $n-b+1$ respectively. Thus, by Lemma $\ref{le:lawsmTMT}$ we have
\begin{align*}\hspace{-15pt}
&\P(\ell_{n,b}>s)\\
&=\binom{n}{b}\P(T\in\mathcal{G})\P(m(T_2)-M(T_1)>s)\\
&=\binom{n}{b}\prod_{i=0}^{n-b-1}\left(\frac{1+i}{b+i}\right)\frac{1}{(b-1)(n-b)}
\int_0^1 \frac{\Psi(b-p)-\Psi(1-p)}{Be(n-b,e^{-s}p)Be(b-1,1-p)} \dif{p} \\
&=\frac{n}{(n-b)b(b-1)}
\int_0^1 \frac{\Psi(b-p)-\Psi(1-p)}{Be(n-b,e^{-s}p)Be(b-1,1-p)} \dif{p}.
\end{align*}
\end{proof}

\begin{proof}[Proof of Corollary $\ref{cor:ConvDistlnb}$]
Observe that, uniformly for $p\in(0,1)$, we have
\begin{equation*}
\Psi(b-p)-\Psi(1-p)=\sum_{k=1}^{b-1}\frac{1}{k-p}=\frac{1}{1-p} + \log b + \bO{1},
\end{equation*}
thus, substituting in \eqref{eq:Llnb} and also using Stirling's approximation and Euler's reflection formula, we obtain
\begin{align*}
\P(\ell_{n,b}>0)\sim& \frac{1}{u(1-u)n}
\int_0^1 \(\frac{1-u}{u}\)^p \frac{\sin \pi p}{\pi} 
\(\frac{1}{1-p}+\log n + \bO{1}\) \dif{p}\\
\sim&\frac{\log n}{u(1-u)n}\int_0^1e^{p\alpha} \frac{\sin \pi p}{\pi}\dif{p}\\
=&\frac{\log n}{u(1-u)n}G(\alpha).
\end{align*}
On the other hand, for any $s>0$ we have
\begin{align*}
\P\(\(\log n\)\ell_{n,b}>s\)\sim& \frac{1}{u(1-u)n}
\int_0^1 \frac{b^{-p}\(n-b\)^{pe^{-s/\log n}}}{\Gamma(1-p)\Gamma(pe^{-s/\log n})}\(\frac{1}{1-p} + \log b + \bO{1}\)\dif{p}\\
\sim&\frac{\log n}{u(1-u)n}\int_0^1 e^{p\alpha}
(n-b)^{p(e^{-s/\log n}-1)}
\frac{1}{\Gamma(1-p)\Gamma(p)}\dif{p}\\
\sim&\frac{\log n}{u(1-u)n}\int_0^1 e^{p\alpha}
(n-b)^{-ps/\log n}
\frac{1}{\Gamma(1-p)\Gamma(p)}\dif{p}\\
\sim&\frac{\log n}{u(1-u)n}\int_0^1 e^{p(\alpha-s)}
\frac{\sin \pi p}{\pi }\dif{p}\\
=&\frac{\log n}{u(1-u)n}G(\alpha-s).
\end{align*}
\end{proof}

\begin{proof}[Proof of Theorem $\ref{le:JointLlnb}$]
Letting $\ell_\pi\coloneqq \L{}\(t:\pi\in\Pi^n(t)\)$ for any subset $\pi\subset [n]$, by exchangeability of $\Pi^n(t)$ we have
\begin{equation*}
\P\(\Lambda_{\mathbf{b},\mathbf{s}}\)=\frac{n!}{b_1!(b_2-b_1)!\cdots(n-b_m)!}
\P\(\bigcap_{1\leq i \leq m}A_{b_i,s_i} , 
\bigcap_{\substack{b>b_1\\b\not\in\mathbf{b}}}\bar A_{b,0} \)
\end{equation*}
where
$$A_{b,s}=\{\ell_{\{1,\dots,b\}}>s\}$$
and
$$\bar A_{b,0}=\{\ell_{\{1,\dots,b\}}=0\}.$$
Recall that $M\(\restr{T}{b_1}\)$ is defined as the maximum of the exponential edges associated to the root of $\restr{T}{b_1}$. Letting $b_{m+1}\coloneqq n$, and also letting $E_b$, $1\leq b \leq n$, be the exponential variable associated to $b$, we have
\begin{align*}
&\hspace{10pt}\P\(\bigcap_{1\leq i \leq m}A_{b_i,s_i} , 
\bigcap_{\substack{b>b_1\\b\not\in\mathbf{b}}}\bar A_{b,0} \)\\
&=\(\prod_{i=1}^{m+1} \frac{1\cdot 2 \cdots (b_{i+1}-b_{i})}{b_i(b_i+1)\cdots(b_{i+1}-1)}\)
\P\(E_{b_1+1}-M\(\restr{T}{b_1}\)>s_1, \bigcap_{i=2}^m
E_{b_i+1}-E_{b_{i-1}+1}>s_i\),
\end{align*}
where the product above is the probability that $T$ is structured in such a way that $\{b_1+1\}$ attaches to $\{1\}$ and is the root 
of a subtree formed with $\{b_1+1,\dots,b_2\}$, that $\{b_2+1\}$ attaches to $\{1\}$ and is the root of a subtree
formed with $\{b_2+1,\dots,b_3\}$, and so forth. Using the independence of the exponential variables we obtain
\begin{align*}
&\hspace{10pt}\P\(E_{b_1+1}-M\(\restr{T}{b_1}\)>s_1, \bigcap_{i=2}^m \{E_{b_i+1}-E_{b_{i-1}+1}>s_i\} \)\\
&=\int_0^\infty \dif{t_1} \int_{t_1+s_1}^\infty \dif{t_2} \dots \int_{t_m+s_m}^\infty \dif{t_{m+1}} 
\(\der{\P\(M(\restr{T}{b_1})\leq t_1\)}{t_1}\) e^{-t_2}\dots e^{-t_{m+1}}\\
&=\int_0^\infty \dif{t_1} \int_{t_1+s_1}^\infty \dif{t_2} \dots \int_{t_{m-1}+s_{m-1}}^\infty \dif{t_m}
\(\der{\P\(M(\restr{T}{b_1})\leq t_1\)}{t_1}\) e^{-t_2}\dots e^{-2t_{m}} e^{-s_m}\\
&\vdots\\
&=\frac{\exp\{-\IP{(m:1)}{\mathbf{s}}\}}{m!}
\int_0^\infty e^{-mt_1}
\der{\P\(M(\restr{T}{b_1})\leq t\)}{t_1}\dif{t_1}.
\end{align*}
From $\eqref{eq:lawMT}$ and making $p=e^{-t}$ in the above integral, and putting all together we obtain $\eqref{eq:JointLlnb}$.
Finally $\eqref{eq:JointLlnb2}$ follows from
$$\P\(\Lambda_{\mathbf{b},\mathbf{s}},\ell_{n,b_1-1}=0\)=
\P\(\Lambda_{\mathbf{b},\mathbf{s}}\)-\P\(\Lambda_{\mathbf{b},\mathbf{s}},\ell_{n,b_1-1}>0\)$$
and, recursively,
\begin{align*}
\P\(\Lambda_{\mathbf{b},\mathbf{s}},\bigcap_{n/2<b<b_1} \{\ell_{n,b}=0\}\)=
\P\(\Lambda_{\mathbf{b},\mathbf{s}}\)-\sum_{n/2<b<b_1}\P\(\Lambda_{\mathbf{b},\mathbf{s}},\ell_{n,b}>0,\bigcap_{i=1}^{b_1-b-1}\{\ell_{n,b+i}=0\}\).
\end{align*}
Substituting $\eqref{eq:JointLlnb}$ in the above expression, we obtain $\eqref{eq:JointLlnb2}$.
\end{proof}

{\bf Aknowledgements}
Alejandro and Arno would like to thank Geronimo Uribe Bravo for his fruitful suggestions on the second moment method.


\end{document}